\documentclass[a4paper,11pt]{article}
\usepackage{graphicx}
\usepackage{amssymb}
\usepackage[french, english]{babel}
\usepackage[utf8]{inputenc}
\usepackage{amsmath,amsfonts,amssymb,amsthm}
\usepackage[T1]{fontenc}
\usepackage{fullpage}
\usepackage{hyperref}
\usepackage{mathrsfs}
\usepackage{relsize}
\usepackage{tikz}
\usepackage{bbm}
\usepackage{appendix}
\usepackage{pifont}
\usepackage[normalem]{ulem}
\usetikzlibrary{patterns}

\definecolor{darkgreen}{rgb}{0.0, 0.5, 0.0}

\newcommand{\eps}{\varepsilon}

\newcommand{\loc}{\mathrm{loc}}

\newcommand{\T}{\mathcal{T}}

\newcommand{\TT}{\mathbb{T}}
\newcommand{\N}{\mathbb N}

\newcommand{\Z}{\mathbb Z}

\newcommand{\E}{\mathbb E}

\renewcommand{\P}{\mathbb P}

\newcommand{\expl}{\mathcal{E}}

\newcommand{\Tii}{\mathbb{T}_{\kappa}^{II}}

\newcommand{\diam}{\mathrm{diam}}

\theoremstyle{definition}

\newtheorem{thm}{Theorem}

\newtheorem{defn}{Definition}%[section]
\newtheorem{rem}[defn]{Remark}

\newtheorem{prop}[defn]{Proposition}
\newtheorem{corr}[defn]{Corollary}

\newtheorem{lem}[defn]{Lemma}

\tikzstyle{every node}=[circle, draw, fill=black!50, inner sep=0pt, minimum width=4pt]
\tikzstyle{white}=[circle, draw, fill=black!0, inner sep=0pt, minimum width=4pt]
\tikzstyle{bigwhite}=[circle, draw, fill=black!0, inner sep=0pt, minimum width=10pt]
\tikzstyle{dual}=[circle, draw=blue, fill=black!0, inner sep=0pt, minimum width=4pt]
\tikzstyle{fat}=[circle, draw, fill=red!50, inner sep=0pt, minimum width=8pt]
\tikzstyle{fat_bis}=[circle, draw, fill=blue!50, inner sep=0pt, minimum width=8pt]
\tikzstyle{fat_ter}=[circle, draw, fill=green!50, inner sep=0pt, minimum width=8pt]
\tikzstyle{rouge}=[circle, draw, fill=red, inner sep=0pt, minimum width=7pt]
\tikzstyle{bleu}=[circle, draw, fill=blue, inner sep=0pt, minimum width=7pt]
\tikzstyle{petitrouge}=[circle, draw, fill=red, inner sep=0pt, minimum width=4pt]
\tikzstyle{petitbleu}=[circle, draw, fill=blue, inner sep=0pt, minimum width=4pt]
\tikzstyle{texte}=[draw=none, fill=none]

\title{\bf{Local limits of uniform triangulations in high genus}}
\author{Thomas \bsc{Budzinski}\footnote{University of British Columbia, \url{budzinski@math.ubc.ca}} \, and Baptiste \bsc{Louf}\footnote{IRIF, Université Paris Diderot, \url{blouf@irif.fr}}}

\begin{document}

\maketitle

\begin{abstract}
We prove a conjecture of Benjamini and Curien stating that the local limits of uniform random triangulations whose genus is proportional to the number of faces are the Planar Stochastic Hyperbolic Triangulations (PSHT) defined in \cite{CurPSHIT}. The proof relies on a combinatorial argument and the Goulden--Jackson recurrence relation to obtain tightness, and probabilistic arguments showing the uniqueness of the limit. As a consequence, we obtain asymptotics up to subexponential factors on the number of triangulations when both the size and the genus go to infinity.

As a part of our proof, we also obtain the following result of independent interest: if a random triangulation of the plane $T$ is weakly Markovian in the sense that the probability to observe a finite triangulation $t$ around the root only depends on the perimeter and volume of $t$, then $T$ is a mixture of PSHT.
\end{abstract}

\section*{Introduction}

\paragraph{Counting maps on surfaces.}
The enumeration of maps or triangulations on surfaces, going back to Tutte \cite{Tutte63} in the planar case, has proved to be connected with many different domains of mathematics and theoretical physics. Such links include the "double scaling limits" considered by physicists in the 90s, the Witten conjecture about the geometry of moduli spaces (\cite{Witten91,Kont92}), the topological recursion (see e.g. \cite{Eynard16}), representations of the symetric group and solutions of several integrable hierarchies such as the KdV, the KP and the $2$-Toda hierarchies \cite{MJD00,Okounkov00,GJ08}. In particular, the link with the KP hierarchy has been used by Goulden and Jackson in \cite{GJ08} to obtain double recurrence formulas on the number of triangulations with size $n$ and genus $g$ (see also \cite{CC15} for similar relations on quadrangulations). However, asymptotics for these numbers are only known when $n \to +\infty$ for $g$ fixed \cite{BC86}, and not when both $n,g \to +\infty$.

\paragraph{Geometric properties of random maps.}
Alongside these enumerative questions, a probabilistic approach has been the object of a lot of study in the last fifteen years: the goal is then to study the geometric properties of a map picked uniformly in a certain class when the genus $g$ and/or the size $n$ become large. In particular, two extreme cases are now pretty well understood.

The first one is the planar case $g=0$, which is understood both through local and scaling limits.
Many natural models of finite random planar maps have been proved to converge locally as their size goes to infinity towards infinite random planar maps such as the UIPT \cite{AS03} (see also \cite{St18} for the type-I UIPT), the UIPQ \cite{Kri05, CD06} or infinite Boltzmann planar maps \cite{Bud15}. On the other hand, the Brownian map \cite{LG11, Mie11} (see also~\cite{CLGplane} for a noncompact version) is now known to be the scaling limit of a wide class of models of random planar maps, see e.g. \cite{AA13, Mar16, ABA19}. The Brownian map is also linked to the Liouville quantum gravity approach \cite{DKRV16}.
Some of these continuous models also have analogues in higher genus such as Brownian surfaces \cite{Bet14} or Liouville quantum gravity on complex tori \cite{DRV16}, but the behaviour of these models when the genus goes to infinity is still poorly understood.

The other extreme case which is well understood is the case where the genus is unconstrained and the maps simply consist of uniform random gluings of polygons. Here the number of vertices tends to be very small and their degrees go to infinity \cite{Gam06, CP16, BCP19}, so no proper local limit exists.

\paragraph{Random maps with genus proportional to the size.}
However, much less is known about the case of maps of higher genus, and in particular when the genus is proportional to the size. The only known results so far are the identification of the local limit of uniform unicellular maps (i.e. maps with one face) \cite{ACCR13}, which is a supercritical random tree, and the calculation of their diameter \cite{Ray13a}.
One of the reasons why it is more difficult to obtain results in high genus is the lack of explicit enumeration results, which play a key role in the planar case.
The goal of this work is to identify the local limit of uniform triangulations whose genus is proportional to the size.

Before describing the limiting objects that appear, let us first explain how the local limit is affected by the genus. By the Euler formula, a triangulation with $2n$ faces and genus $g$ has $3n$ edges and $n+2-2g$ vertices, which implies $g \leq \frac{n+1}{2}$. Hence, if $\frac{g}{n} \to \theta \in \left[ 0, \frac{1}{2} \right]$, then the average degree of the vertices goes to $\frac{6}{1-2\theta}$. In particular, if $0<\theta<\frac{1}{2}$, this mean degree lies strictly between $6$ and $+\infty$. Therefore, it is natural to expect limit objects to be hyperbolic triangulations of the plane\footnote{As a deterministic example, the $d$-regular triangulations of the plane for $d>6$ are hyperbolic.}. This expected relation between high genus and hyperbolic objects also echoes the construction of higher genus surfaces from the hyperbolic plane in complex geometry.

\paragraph{Planar Stochastic Hyperbolic Triangulations.}
This has motivated the introduction of random hyperbolic triangulations, first in the half-planar case by Angel and Ray \cite{AR13}, and then in the full-plane case by Curien \cite{CurPSHIT}. More precisely, Curien built a one-parameter family $(\TT_{\lambda})_{0 < \lambda \leq \lambda_c}$ of random triangulations of the plane\footnote{To be exact, the triangulations defined in \cite{CurPSHIT} are type-II triangulations, i.e. triangulations with no loop joining a vertex to itself. The type-I (with loops) analogue, which will be the one considered in this work, was defined in \cite{B16}.}, where $\lambda_c=\frac{1}{12\sqrt{3}}$, and characterized them as the only random triangulations of the plane exhibiting a natural spatial Markov property. For any finite triangulation $t$ with a hole of perimeter $p$ and $v$ vertices in total, we have
\[ \P \left( t \subset \TT_{\lambda} \right)=C_p(\lambda) \lambda^v,\]
where $C_p(\lambda)$ are explicit functions of $\lambda$ and, by $t \subset T$, we mean that $T$ can be obtained by filling the hole of $t$ with an infinite triangulation. Moreover, $\TT_{\lambda_c}$ is the UIPT, whereas for $\lambda<\lambda_c$, the map $\TT_{\lambda}$ has hyperbolicity properties such as exponential volume growth \cite{Ray13, CurPSHIT}, positive speed of the simple random walk \cite{CurPSHIT, ANR14} or the existence of a lot of infinite geodesics escaping quickly away from each other \cite{B18}.

\paragraph{The PSHT as local limits.}
For any $g \geq 0$ and $n \geq 2g-1$, we denote by $\T(n,g)$ the set of rooted type-I triangulations of genus $g$ with $2n$ faces. By \emph{rooted}, we mean that the triangulation is equipped with a distinguished oriented edge called the \emph{root}. Let also $T_{n,g}$ be a uniform triangulation of $\T(n,g)$. We also recall that a sequence of rooted triangulations $(t_n)$ converges locally to a triangulation $T$ if for any $r \geq 0$, the ball of radius $r$ around the root in $t_n$, seen as a map, converges to the ball of radius $r$ in $T$. We refer to Section 1 for more precise definitions. For any $\lambda \in (0,\lambda_c]$, let $h \in \left( 0,\frac{1}{4} \right]$ be such that $\lambda=\frac{h}{(1+8h)^{3/2}}$, and let
\begin{equation}
d(\lambda)=\frac{  h \log \frac { 1 + \sqrt { 1 - 4 h } } { 1 - \sqrt { 1 - 4 h } } } { ( 1 + 8 h ) \sqrt { 1 - 4 h } }.
\end{equation}
It can be checked that the function $d(\lambda)$ is increasing with $d(\lambda_c)=\frac{1}{6}$ and $\lim_{\lambda \to 0} d(\lambda)=0$ (see the end of Section 4.3 for a quick proof). Then our main result is the following.
\begin{thm}\label{main_thm}
Let $(g_n)$ be a sequence such that $\frac{g_n}{n} \to \theta$ with $\theta \in \left[ 0,\frac{1}{2} \right)$. Then we have
\[ T_{n,g_n} \xrightarrow[n \to +\infty]{(d)} \TT_{\lambda}\]
for the local topology, where $\lambda$ is the unique solution to the equation
\begin{equation}\label{eqn_lambda_vs_theta}
d(\lambda)=\frac{1-2\theta}{6}.
\end{equation}
\end{thm}
We highlight that we only prove this theorem for type-I triangulations.
This result was conjectured by Benjamini and Curien \cite{CurPSHIT} (in the type-II case) without an explicit formula for $d(\lambda)$, and the formula for $d(\lambda)$ was first conjectured in \cite[Appendix B]{B18these}. The reason why the formula~\eqref{eqn_lambda_vs_theta} appears is that $d(\lambda)$ is the expected inverse of the root degree in $\TT_{\lambda}$, while the corresponding quantity in $T_{n,g_n}$ is asymptotically $\frac{1-2\theta}{6}$ by the Euler formula.
While it may seem counter-intuitive that high genus objects yield planar maps in the local limit, this has already been proved for other models such as random regular graphs or unicellular maps \cite{ACCR13}. Note that the case $\theta=0$ corresponds to $\lambda=\lambda_c$, which proves that if $g_n=o(n)$, then $T_{n,g_n}$ converges to the UIPT, which also seems to be a new result, even for $g_n$ constant. On the other hand, when $\theta \to \frac{1}{2}$, we have $\lambda \to 0$, so all the range $(0,\lambda_c]$ is covered. Since the object $\TT_0$ is not well defined (it corresponds to a "triangulation" where the vertex degrees are infinite), we expect that if $\theta=\frac{1}{2}$, the sequence $(T_{n,g_n})$ is not tight for the local topology.

\paragraph{Strategy of the proof.}
The most natural idea to prove Theorem \ref{main_thm} would be to obtain precise asymptotics for the numbers $\tau(n,g)=|\T(n,g)|$ and to adapt the ideas of \cite{AS03}. In theory, these numbers are entirely characterized by the Goulden--Jackson recurrence equation \cite{GJ08}. However, this seems very difficult without any a priori estimate on the $\tau(n,g)$ and all our efforts to extract asymptotics when $\frac{g}{n} \to \theta>0$ from these relations have failed. Therefore, our proof relies on more probabilistic considerations. It is however interesting to note that our probabilistic arguments allow in the end to obtain combinatorial asymptotics (Theorem \ref{thm_asympto}).

The first part of the proof consists of a tightness result: we prove that $(T_{n,g_n})$ is tight for the local topology as long as $\frac{g_n}{n}$ stays bounded away from $\frac{1}{2}$. A key tool in the proof is the \emph{bounded ratio lemma} (Lemma \ref{ratio_lem}), which states that the ratio $\frac{\tau(n+1,g)}{\tau(n,g)}$ is bounded as long as $\frac{g}{n}$ stays bounded away from $\frac{1}{2}$. This is essentially enough to adapt the argument of Angel and Schramm \cite{AS03} for the tightness of $T_{n,0}$. Along the way, we also show that any subsequential limit is a.s. planar and one-ended. The Goulden--Jackson formula also plays an important role in the proof.

The next step is to notice that any subsequential limit $T$ satisfies a weak Markov property: if $t$ is a finite triangulation with a hole of perimeter $p$ and $v$ vertices in total, then $\P \left( t \subset T \right)$ only depends on $p$ and $v$. From here, we deduce that $T$ must be a mixture of PSHT, i.e. a PSHT with a random parameter $\Lambda$.

Finally, what is left to prove is that $\Lambda$ is deterministic, i.e. it does not depend on $T_{n,g_n}$. By a surgery argument on finite triangulations which we call the \emph{two holes argument}, we first show that if $T_{n,g_n}$ is fixed, then $\Lambda$ does not depend on the choice of the root. We conclude by using the fact that the average inverse degree of the root in $T_{n,g_n}$ is asymptotically $\frac{1-2\theta}{6}$.

\paragraph{Weakly Markovian triangulations.}
Since one of the steps of the proof is a result of independent interest, let us highlight it right now. We call a random triangulation of the plane $T$ \emph{weakly Markovian} if for any finite triangulation $t$ with a hole of perimeter $p$ and $v$ vertices in total, the probability $\P \left( t \subset T \right)$ only depends on $p$ and $v$. This is strictly weaker than the spatial Markov property considered in \cite{CurPSHIT} to define the PSHT, since any mixture of PSHT is weakly Markovian. The result we prove is the following.

\begin{thm}\label{thm_weak_markov}
Any weakly Markovian triangulation of the plane is a mixture of PSHT.
\end{thm}

\paragraph{Combinatorial asymptotics.}
Finally, while we were unable to obtain directly asymptotics on $\tau(n,g)$ when both $n$ and $g$ go to $+\infty$, Theorem \ref{main_thm} allows us to obtain such estimates up to sub-exponential factors. For any $\theta \in \left[ 0,\frac{1}{2} \right)$, we denote by $\lambda(\theta)$ the value of $\lambda$ given by \eqref{eqn_lambda_vs_theta}.

\begin{thm}\label{thm_asympto}
Let $(g_n)$ be a sequence such that $0 \leq g_n \leq \frac{n+1}{2}$ for every $n$ and $\frac{g_n}{n} \to \theta \in \left[0, \frac{1}{2} \right]$. Then we have
\[ \tau(n,g_n)=n^{2g_n} \exp \left( f(\theta) n + o(n) \right) \]
as $n \to +\infty$, where $f(0)=\log 12\sqrt{3}$, also $f(1/2)=\log \frac{6}{e}$ and
\begin{equation}\label{eq_f}
f(\theta)= 2 \theta\log \frac{12\theta}{e} + \theta\int_{2}^{1/\theta} \log \frac{1}{\lambda(1/t)}\mathrm{d}t
\end{equation}
for $0<\theta<\frac{1}{2}$.
\end{thm}
To the best of our knowledge, these are the first asymptotic results on the number of triangulations with both large size and high genus.
Note that the integral is well defined since $\lambda(\theta)$ is a continuous function and we have $\lambda(\theta) =O \left( 1/2-\theta \right)$ when $\theta \to 1/2$. Moreover, since $\lambda(\theta) \to \frac{1}{12\sqrt{3}}$ as $\theta \to 0$, it is easy to see that the function $f$ is continuous at $0$ and at $1/2$. The proof mostly relies on the observation that Theorem \ref{main_thm} gives the limit values of the ratio $\frac{\tau(n+1,g)}{\tau(n,g)}$.

\paragraph{Other types of triangulations.}
A natural question, which we do not answer in this paper, is to ask whether Theorem~\ref{main_thm} can be extended to type-II (i.e. with multiple edges but no loop) or type-III (i.e. with neither multiple edges nor loops) triangulations. To adapt our argument in the type-II setting, one would need to be extra careful with the surgery operations of Section~\ref{sec_tightness} and to overcome the absence of a Goulden--Jackson formula. The question seems more complicated for type-III triangulations, since then the spatial Markov property is partly lost, and it is not even clear how to define the PSHT. Another natural strategy would be to deduce type-II (resp. type-III) results from Theorem~\ref{main_thm}. The first step would be to prove that $T_{n,g}$ has a large $2$-connected (resp. $3$-connected) core.

\paragraph{Structure of the paper.}
The structure of the paper is as follows. In Section 1, we review basic definitions and previous results that will be used throughout the paper. In Section 2, we prove that the triangulations $T_{n,g_n}$ are tight for the local topology, and that any subsequential limit is a.s. planar and one-ended. In Section 3, we prove Theorem \ref{thm_weak_markov}, which implies that any subsequential limit of $T_{n,g_n}$ is a PSHT with random parameter $\Lambda$. In Section 4, we conclude the proof of Theorem \ref{main_thm} by showing that $\Lambda$ is deterministic and depends only on $\theta$. Finally, Section 5 is devoted to the proof of Theorem \ref{thm_asympto}.

\paragraph{Acknowledgments.} The authors thank Guillaume Chapuy and Nicolas Curien for helpful discussions and comments on earlier versions of this manuscript. The authors also thank the two anonymous referees for useful remarks. The first author is supported by ERC Geobrown (740943). The second author is fully supported by ERC-2016-STG 716083 "CombiTop". The authors would also like to thank the Isaac Newton Institute for Mathematical Sciences (EPSRC grant number EP/R014604/1) for its hospitality during the Random Geometry follow-up workshop when this work was started.

\tableofcontents

\section{Preliminaries}

\subsection{Definitions}

The goal of this paragraph is to state basic definitions on triangulations that will be used throughout the paper.

As in~\cite{C-StFlour}, we define a (finite or infinite) \emph{map} $M$ as a way to glue a collection of oriented polygons, called the \emph{faces}, along their edges in a connected way that matches the orientations. Note that this definition is not restricted to maps with finitely many faces. By forgetting the faces of $M$ and looking only at its vertices and edges, we obtain a graph $G$ (if $M$ is infinite, then $G$ may have vertices with infinite degree).

If the number of polygons is finite, then $M$ is always homeomorphic to an orientable topological surface, so we can define the genus of $M$ as the genus of this surface. The maps that we consider will always be \emph{rooted}, i.e. equipped with a distinguished oriented edge called the \emph{root edge}. The face on the right of the root edge is the \emph{root face}, and the vertex at the start of the root edge is the \emph{root vertex}.

A \emph{triangulation} is a rooted map where all the faces have degree $3$. We will mostly be interested in \emph{type-I triangulations}, i.e. triangulations that may contains loops and multiple edges. We mention right now that a \emph{type-II triangulation} is a triangulation that may contain multiple edges, but no loops. In graph-theoretic terms, a type-$i$ triangulation is a triangulation with girth (i.e. smallest cycle length) at least $i$. Unless specified otherwise, by \emph{triangulation}, we will always mean \emph{type-I triangulation}.

For every $n \geq 1$ and $g \geq 0$, we will denote by $\T(n,g)$ the set of triangulations of genus $g$ with $2n$ faces (the number of faces must be even to glue the edges two by two). By the Euler formula, a triangulation of $\T(n,g)$ has $3n$ edges and $n+2-2g$ vertices. In particular, the set $\T(n,g)$ is nonempty if and only if $n \geq 2g-1$. We will also denote by $\tau(n,g)$ the cardinal of $\T(n,g)$ and by $T_{n,g}$ a uniform random variable on $\T(n,g)$.

We will also need to consider two different notions of triangulations with boundaries, that we call \emph{triangulations with holes} and \emph{triangulations of multi-polygons}. Basically, the first ones will be used to describe a neighbourhood of the root in a triangulation, and the second ones to describe the complementary of this neighbourhood.

For $\ell \geq 1$ and $p_1, p_2, \dots, p_{\ell} \geq 1$, we call a \emph{triangulation with holes of perimeter $p_1, \dots, p_{\ell}$} a map where all the faces have degree $3$ except, for every $1 \leq i \leq \ell$, a face $h_i$ of degree $p_i$. The faces $h_i$ are called the \emph{holes}. The boundaries of the faces $h_i$ must be simple and edge-disjoint, but may have common vertices (see the bottom part of Figure \ref{fig_t_subset_T}). A triangulation with holes will be rooted at a distinguished oriented edge, which may lie on the boundary of a hole or not. Triangulations with holes will always be finite.

A (possibly infinite) \emph{triangulation of the $(p_1, \dots, p_{\ell})$-gon} is a map where all the faces have degree $3$ except, for every $1 \leq i \leq \ell$, a face $f_i$ of degree $p_i$. The faces $f_i$ are called the \emph{external faces}, and must be simple and have vertex-disjoint boundaries. Moreover, each of the external faces comes with a distinguished edge on its boundary, such that the external face lies on the right of the distinguished edge.

We denote by $\T_{p_1,p_2,\dots,p_{\ell}}(n,g)$ the set of triangulations of the $(p_1,p_2,\dots,p_{\ell})$-gon of genus $g$ with $2n-\sum_{i=1}^{\ell} (p_i-2)$ triangles, and by $\tau_{p_1,p_2,\dots,p_{\ell}}(n,g)$ its cardinal. The reason why we choose this convention is that by the Euler formula, a triangulation of $\T_{p_1,p_2,\dots,p_{\ell}}(n,g)$ has $n+2-2g$ vertices in total, just like a triangulation of $\T(n,g)$.

If $t$ is a triangulation with holes and $T$ a (finite or infinite) triangulation, we write $t \subset T$ if $T$ can be obtained from $t$ by gluing one or several triangulations of multi-polygons to the holes of $t$ (see Figure \ref{fig_t_subset_T}). In particular, in the planar case, this definition coincides with the one used e.g. in \cite{AS03}. If $T$ is an infinite triangulation, we say that it is \emph{one-ended} if for every finite $t$ with $t \subset T$, only one connected component of $T \backslash t$ contains infinitely many triangles. We also say that $T$ is \emph{planar} if every finite $t$ with $t \subset T$ is planar.

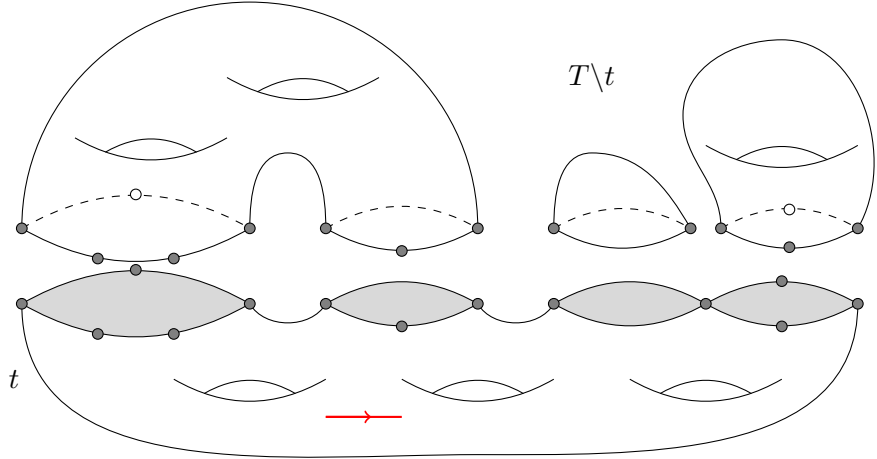
\begin{figure}
\begin{center}
\begin{tikzpicture}
\fill[gray!30] (-6,0) to[bend left] (-3,0) to[bend left] (-6,0);
\fill[gray!30] (-2,0) to[bend left] (0,0) to[bend left] (-2,0);
\fill[gray!30] (1,0) to[bend left] (3,0) to[bend left] (1,0);
\fill[gray!30] (3,0) to[bend left] (5,0) to[bend left] (3,0);

\draw(-6,0) to[out=270,in=180] (0,-2) to [out=0,in=270] (5,0);
\draw(-6,0) to[bend left] (-3,0) to[bend left] (-6,0);
\draw(-2,0) to[bend left] (0,0) to[bend left] (-2,0);
\draw(1,0) to[bend left] (3,0) to[bend left] (1,0);
\draw(3,0) to[bend left] (5,0) to[bend left] (3,0);
\draw(-3,0)to[bend right=60](-2,0);
\draw(0,0)to[bend right=60](1,0);

\draw(-4,-1)to[bend right](-2,-1);
\draw(-3.6,-1.19)to[bend left](-2.4,-1.19);
\draw(-1,-1)to[bend right](1,-1);
\draw(-0.6,-1.19)to[bend left](0.6,-1.19);
\draw(2,-1)to[bend right](4,-1);
\draw(2.4,-1.19)to[bend left](3.6,-1.19);
\draw(-5.3,2.2)to[bend right](-3.3,2.2);
\draw(-4.9,2.01)to[bend left](-3.7,2.01);
\draw(-3.3,3)to[bend right](-1.3,3);
\draw(-2.9,2.81)to[bend left](-1.7,2.81);
\draw(3,2.1)to[bend right](5,2.1);
\draw(3.4,1.91)to[bend left](4.6,1.91);

\draw(-6,1) to[bend right] (-3,1);
\draw[dashed](-6,1) to[bend left] (-3,1);
\draw(-2,1) to[bend right] (0,1);
\draw(1,1) to[bend right] (2.8,1);
\draw(3.2,1) to[bend right] (5,1);
\draw[dashed](-2,1) to[bend left] (0,1);
\draw[dashed](1,1) to[bend left] (2.8,1);
\draw[dashed](3.2,1) to[bend left] (5,1);

\draw(-3,1)to[out=90,in=180](-2.5,2)to[out=0,in=90](-2,1);
\draw(-6,1) to[out=90,in=180] (-3,4) to[out=0,in=90] (0,1);
\draw(1,1) to[out=90,in=180] (1.5,2) to[out=0,in=120] (2.8,1);
\draw(3.2,1) to[out=90, in=270] (2.7,2.5) to[out=90,in=180] (4,3.5) to[out=0,in=60] (5,1);

\draw[red, thick, ->](-2,-1.5)--(-1.4,-1.5);
\draw[red, thick](-2,-1.5)--(-1,-1.5);
\draw(-6.1,-1) node[texte]{$t$};
\draw(1.5,3) node[texte]{$T \backslash t$};

\draw(-6,0)node{};
\draw(-5,-0.4)node{};
\draw(-4,-0.4)node{};
\draw(-3,0)node{};
\draw(-4.5,0.45)node{};
\draw(-2,0)node{};
\draw(0,0)node{};
\draw(-1,-0.3)node{};
\draw(1,0)node{};
\draw(3,0)node{};
\draw(5,0)node{};
\draw(4,-0.3)node{};
\draw(4,0.3)node{};

\draw(-6,1)node{};
\draw(-5,0.6)node{};
\draw(-4,0.6)node{};
\draw(-3,1)node{};
\draw(-4.5,1.45)node[white]{};
\draw(-2,1)node{};
\draw(0,1)node{};
\draw(-1,0.7)node{};
\draw(1,1)node{};
\draw(2.8,1)node{};
\draw(3.2,1)node{};
\draw(5,1)node{};
\draw(4.1,0.75)node{};
\draw(4.1,1.25)node[white]{};
\end{tikzpicture}
\end{center}
\caption{A large triangulation $T$ with genus $7$ and a smaller triangulation with holes $t$ (in the bottom) such that $t \subset T$. The holes are filled by a triangulation of the $(3,5)$-gon, a triangulation of the $2$-gon and a triangulation of the $4$-gon. The triangles are not drawn on the picture.}\label{fig_t_subset_T}
\end{figure}

We also recall that to a triangulation $t$, we can naturally associate its \emph{dual map} $t^*$: it is the map whose vertices are the faces of $t$ and where for each edge $e$ of $t$, we draw the dual edge $e^*$ joining the two faces incident to $e$. If $t$ is a triangulation of a multi-polygon, it will be more suitable to work with the convention that the external faces \emph{do not} belong to the dual $t^*$. Note that triangulations of multi-polygons have simple and disjoint boundaries, so their dual $t^*$ will always be connected.

Finally, we recall the definition of the \emph{graph distance} in a map. For a pair of vertices $(v,v')$, the distance $d_t(v,v')$ is the length of the shortest path of edges of $t$ between $v$ and $v'$. We call $d_t^*$ the graph distance in the dual\footnote{In particular, if $t$ is a triangulation of a multi-polygon, then $d_t^*(f,f')$ is the length of the smallest dual path which \emph{avoids} the external faces.} map $t^*$. We also note that there is a natural way to extend $d^*_t$ to the vertices of $t$.
For a pair of distinct vertices $(v,v')$, we set
\[d_t^*(v,v')=\min(d_t^*(f,f'))+1,\]
where the minimum is taken over all pairs $(f,f')$ of faces such that $f$ is incident to $v$ and $f'$ is incident to $v'$.

\subsection{Combinatorics}

The goal of this paragraph is to summarize some previously known or basic combinatorial results about triangulations in higher genus. We start with the Goulden--Jackson recurrence formula.

\begin{thm}\cite{GJ08}
Let $f(n,g)=(3n+2) \tau(n,g)$, with the conventions $f(-1,0)=\frac{1}{2}$, $f(0,0)=2$ and $f(-1,g)=f(0,g)=0$ for $g \geq 1$. For every $n,g \geq 0$ with $g \leq \frac{n+1}{2}$, we have
\begin{equation}\label{eqn_Goulden_Jackson}
f(n,g)=\frac{4(3n+2)}{n+1} \left( n(3n-2) f(n-2,g-1) + \sum_{\substack{n_1+n_2=n-2 \\ g_1+g_2=g}} f(n_1, g_1) f(n_2, g_2) \right).
\end{equation}
\end{thm}
As explained in the introduction, this formula is in theory enough to compute all the cardinals $\tau(n,g)$, but efforts to extract asymptotics from here when both $n$ and $g$ go to $+\infty$ have failed so far. Our only use of this formula willl be in Section~\ref{subsec_planar_one}. We will not fully use the Goulden--Jackson formula, but only the two inequalities~\eqref{eqn_GJ_consequence_1} and~\eqref{eqn_GJ_consequence_2} that both follow easily from~\eqref{eqn_Goulden_Jackson}.

We also state right now a crude inequality that bounds the number of triangulations of multi-polygons with genus $g$ by the number of triangulations of genus $g$. This will be useful later.

\begin{lem}\label{lem_fill_holes}
For every $n,g \geq 0$ and $p_1, \dots, p_{\ell} \geq 1$, we have
\[ \tau_{p_1, \dots, p_{\ell}} (n,g) \leq (6n)^{\ell-1} \tau(n,g). \]
\end{lem}

\begin{proof}
We describe a way to associate with each map $t$ of $\T_{p_1, \dots, p_{\ell}}(n,g)$ a map $\widetilde{t}$ of $\T(n,g)$ with some marked oriented edges. For each external face $f_i$ of $t$:
\begin{itemize}
\item[$\bullet$]
if $p_i \geq 3$, we triangulate $f_i$ by joining all the vertices of $\partial f_i$ to the start of the distinguished edge on $\partial f_i$, and we mark this edge as $e_i$;
\item[$\bullet$]
if $p_i=2$, we simply glue together the two edges of $\partial f_i$, and mark the edge that we obtain as $e_i$;
\item[$\bullet$]
if $p_i=1$, we use the "classical" root transformation shown on Figure \ref{boundary_one}, and mark the edge obtained by the gluing as $e_i$.
\end{itemize}
We obtain a triangulation with the same genus as the initial one, and we root it at $e_1$. Note that the above operation does not change the number of vertices, so the triangulation belongs to $\T(n,g)$.
It is easy to see that $t \mapsto \widetilde{t}$ is injective. Indeed, if we know $p_i \geq 3$ and the edge $e_i$, then the $p_i-2$ triangles created by triangulating $f_i$ are the first $p_i-2$ triangles on the right of $e_i$ that are incident to its starting point. If $p_i \in \{1,2\}$, the reverse operation is straightforward. Finally, $\widetilde{t}$ is a triangulation of $\T(n,g)$ with $\ell-1$ marked oriented edges (plus its root edge). Since any triangulation of $\T(n,g)$ has $6n$ oriented edges, we are done.

\end{proof}
\begin{figure}%[!h]
\centering
\includegraphics[scale=0.5]{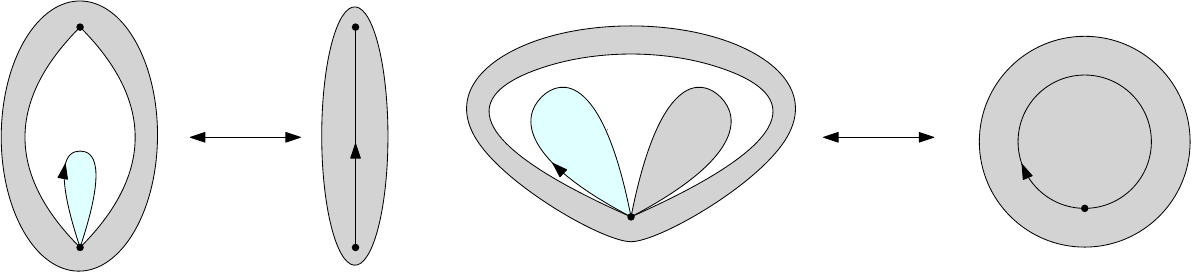}
\caption{Getting rid of a boundary of size $1$. The boundary faces are in blue}
\label{boundary_one}
\end{figure}
\begin{rem}\label{rem_root_transfo}
The bijection of Figure \ref{boundary_one} is classical and implies in particular $\tau_1(n,g)=\tau(n,g)$.
\end{rem}

\subsection{The PSHT}
\label{subsec_psht}

In this subsection, we recall the definition and some basic properties of the type-I Planar Stochastic Hyperbolic Triangulations, or PSHT. They were introduced in \cite{CurPSHIT} in the type-II setting (no loops), but we will be more interested in the type-I PSHT defined in \cite{B16}. The PSHT $(\TT_{\lambda})_{0 < \lambda \leq \lambda_c}$ form a one-parameter family of random infinite triangulations of the plane, where $\lambda_c=\frac{1}{12\sqrt{3}}$. Their distribution is characterized as follows. There is a family of constants $\left( C_p(\lambda) \right)_{p \geq 1}$ such that for every planar triangulation with a hole of perimeter $p$ and $v$ vertices in total, we have
\begin{equation}\label{definition_psht}
\P \left( t \subset \TT_{\lambda} \right) =C_p(\lambda) \times \lambda^{v}.
\end{equation}
Moreover, let $h$ be the unique solution in $\left( 0, \frac{1}{4} \right]$ of
\begin{equation}\label{eqn_defn_h}
\lambda=\frac{h}{(1+8h)^{3/2}}.
\end{equation}
Then we have
\begin{equation}\label{eqn_cp_psht}
C_p(\lambda)=\frac{1}{\lambda} \left( 8+\frac{1}{h} \right)^{p-1} \sum_{q=0}^{p-1} \binom{2q}{q} h^q,
\end{equation}
so the distribution of $\TT_{\lambda}$ is completely explicit.

A very useful consequence of \eqref{definition_psht} is the spatial Markov property of $\TT_{\lambda}$: for any triangulation $t$ with a hole of perimeter $p$, conditionally on $t \subset \TT_{\lambda}$, the distribution of the complementary $\TT_{\lambda} \backslash t$ only depends on $p$. Therefore, it is possible to discover it in a Markovian way by a \emph{peeling exploration}.

Since this will be useful later, we recall basic definitions related to peeling explorations. A \emph{peeling algorithm} $\mathcal{A}$ is a mapping that associates with every triangulation with holes an edge on the boundary of one of the holes. Given an infinite triangulation $T$ and a peeling algorithm $\mathcal{A}$, we can define an increasing sequence $\left( \expl_T^{\mathcal{A}}(k) \right)_{k \geq 0}$ of triangulations with holes such that $\expl_T^{\mathcal{A}}(k) \subset T$ for every $k$ in the following way:
\begin{itemize}
\item
the map $\expl_T^{\mathcal{A}}(0)$ is the trivial map consisting of the root edge only,
\item
for every $k \geq 1$, the triangulation $\expl_{T}^{\mathcal{A}}(k+1)$ is obtained from $\expl_T^{\mathcal{A}}(k)$ by adding the triangle incident to $\mathcal{A} \left( \expl_T^{\mathcal{A}}(k) \right)$ outside of $\expl_T^{\mathcal{A}}(k)$ and, if this triangle creates a finite hole, all the triangles in this hole.
\end{itemize}
Such an exploration is called \emph{filled-in}, because all the finite holes are filled at each step. For the PSHT, we denote by $P^{\lambda}(k)$ and $V^{\lambda}(k)$ the perimeter and volume of $\expl_T^{\mathcal{A}}(k)$, where by \emph{volume} we mean \emph{total number of vertices}. The spatial Markov property ensures that $\left( P^{\lambda}(k), V^{\lambda}(k) \right)_{k \geq 0}$ is a Markov chain on $\N^2$ and that its transitions do not depend on the algorithm $\mathcal{A}$. We also recall the asymptotic behaviour of these two processes. We have
\begin{equation}\label{eqn_asymptotic_peeling_psht}
\frac{P^{\lambda}(k)}{k} \xrightarrow[k \to +\infty]{a.s.}\sqrt{\frac{1-4h}{1+8h}} \hspace{1cm} \mbox{and} \hspace{1cm} \frac{V^{\lambda}(k)}{k} \xrightarrow[k \to +\infty]{a.s.} \frac{1}{\sqrt{(1+8h)(1-4h)}},
\end{equation}
where $h$ is given by \eqref{eqn_defn_h}. These estimates are proved in \cite{CurPSHIT} in the type-II setting. For the type-I PSHT, the proofs are the same and use the combinatorial results of \cite{Kri07}. In particular, the asymptotic ratio between $P^{\lambda}(k)$ and $V^{\lambda}(k)$ is $1-4h$, which is a decreasing function of $\lambda$. This shows that the PSHT for different values of $\lambda$ are singular with respect to each other, which will be useful later.

\subsection{Local convergence and dual local convergence.}
\label{subsec_local_topology}

The goal of this section is to recall the definition of local convergence in a setting that is not restricted to planar maps. We also define a weaker (at least for triangulations) notion of local convergence that we call "dual local convergence".

As in the planar case, to define the local convergence, we first need to define balls of triangulations. Let $t$ be a finite triangulation. As usual, for every $r \geq 1$, we denote by $B_r(t)$ the map formed by all the faces of $t$ which are incident to at least one vertex at distance at most $r-1$ from the root vertex, along with all their vertices and edges. We denote by $\partial B_r(t)$ the set of edges $e$ such that exactly one side of $e$ is adjacent to a triangle of $B_r(t)$. The other sides of these edges form a finite number of holes, so $B_r(t)$ is a finite triangulation with holes. Note that contrary to the planar case, there is no bijection between the holes and the connected components of $t \backslash B_r(t)$ (cf. the component on the left of Figure \ref{fig_t_subset_T}). We also write $B_0(t)$ for the trivial "map" consisting of only one vertex and zero edge.

For any two finite triangulations $t$ and $t'$, we write
\[d_{\loc}(t,t')=\left( 1+\max \{r \geq 0 | B_r(t)=B_r(t')\} \right)^{-1}.\]
This is the \emph{local distance} on the set of finite triangulations. As in the planar case, its completion $\overline{\T}$ is a Polish space, which can be viewed as the set of (finite or infinite) triangulations in which all the vertices have finite degrees. However, this space is not compact.

In some parts of this paper, it will be more convenient to work with a weaker notion of convergence which we call the \emph{dual local convergence}. The reason for this is that, since the degrees in the dual of a triangulation are bounded by $3$, tightness for this distance will be immediate, which will allow us to work directly on infinite subsequential limits.

More precisely, we recall that $d^*$ is the graph distance on the \emph{dual} of a triangulation. For any finite triangulation $t$ and any $r \geq 0$, we denote by $B_r^{*}(t)$ the map formed by all the faces at dual distance at most $r$ from the root face, along with all their vertices and edges. Like $B_r(t)$, this is a finite triangulation with holes. For any two finite triangulations $t$ and $t'$, we write
\[d_{\loc}^*(t,t')=\left( 1+\max \{r \geq 0 | B_r^*(t)=B_r^*(t')\} \right)^{-1}.\]
Note that in any triangulation $t$, since the dual graph of $t$ is $3$-regular, there are at most $3 \times 2^{r-1}$ faces at distance $r$ from the root face. Therefore, for each $r$, the volume of $B_r^*(t)$ is bounded by a constant depending only on $r$, so $B_r^*(t)$ can only take finitely many values. It follows from a simple diagonal extraction argument that the completion for $d_{\loc}^*$ of the set of finite triangulations is compact. We write it $\overline{\T}^*$. This set coincides with the set of finite or infinite triangulations, where the degrees of the vertices may be infinite.

Roughly speaking, the main steps of our proof for tightness will be the following. Since $(\overline{\T}^*, d_{\loc}^*)$ is compact, the sequence $(T_{n, g_n})$ is tight for $d^*_{\loc}$. We will prove that every subsequential limit is planar and one-ended, and finally that its vertices must have finite degrees. We state right now an easy, deterministic lemma that will allow us to conclude at this point.

\begin{lem}\label{lem_dual_convergence}
Let $(t_n)$ be a sequence of triangulations of $\overline{\T}$. Assume that
\[ t_n \xrightarrow[n \to +\infty]{d_{\loc}^*} t,\]
with $t \in \overline{\T}$. Then $t_n \to t$ for $d_{\loc}$ when $n \to +\infty$.
\end{lem}

Note that the converse is very easy: the dual ball $B_r^*(t)$ is a deterministic function of $B_{r+1}(t)$, so $d^*_{\loc} \leq 2 d_{\loc}$, and convergence for $d_{\loc}$ implies convergence for $d^*_{\loc}$.

\begin{proof}[Proof of Lemma \ref{lem_dual_convergence}]
Let $r \geq 1$. Since $t \in \overline{\T}$, the ball $B_r(t)$ is finite, so we can find $r^*$ such that $B_r(t) \subset B_{r^*}^*(t)$. By definition of $d_{\loc}^*$, for $n$ large enough, we have $B^*_{r^*}(t_n)=B^*_{r^*}(t)$. Therefore, we have $B_r (t) \subset t_n$, so $B_r(t_n)=B_r(t)$ for $n$ large enough. Since this is true for any $r \geq 1$, we are done.
\end{proof}

\section{Tightness, planarity and one-endedness}
\label{sec_tightness}

\subsection{The bounded ratio lemma}

The goal of this section is to prove the following result, which will be our main new input in the proof of tightness.
\begin{lem}[Bounded ratio lemma]\label{ratio_lem}
Let $\eps>0$. Then there is a constant $C_\eps>0$ with the following property: for every $n,g \geq 0$ satisfying $\frac{g}{n} \leq\frac{1}{2}-\eps$ and for every $p \geq 1$, we have
\[\frac{\tau_p(n,g)}{\tau_p(n-1,g)} \leq C_\eps.\]
In particular, by the usual bijection between $\T_1(n,g)$ and $\T(n,g)$, we have
\[\frac{\tau(n,g)}{\tau(n-1,g)} \leq C_\eps.\]
\end{lem}

For our future use, it will be important that the constant $C_{\eps}$ does not depend on $p$. The idea of the proof of Lemma \ref{ratio_lem} will be to find an "almost-injective" way to obtain a triangulation of $\T_p(n-1,g)$ from a triangulation of $\T_p(n,g)$. This will be done by merging two vertices together. For this, it will be useful to find two vertices that are quite close to each other and have a reasonnable degree. This is the point of the next result. We recall that for two vertices $v,v'$, the distance $d^*(v,v')$ is the length of the smallest dual path from $v$ to $v'$ that avoids the external faces. Let us fix $\eps>0$. We will call a pair $(v_1, v_2)$ of vertices \emph{good} if $\deg(v_{1})+\deg(v_{2})\leq \frac{12}{\eps}$ and $d^*(v_1,v_2) \leq \frac{24}{\eps}$.

\begin{lem}\label{lem_good_pairs}
In any triangulation of a polygon $t \in \mathcal{T}_p(n,g)$ with $\frac{g}{n} \leq \frac{1}{2}-\eps$, there are at least $\frac{\eps}{12} n$ good pairs of vertices.
\end{lem}

\begin{proof}
Fix a triangulation $t \in \T_p(n,g)$ with $\frac{g}{n} \leq \frac{1}{2}-\eps$. We first note that a positive proportion of the vertices have a small degree. Indeed, by the Euler formula, $t$ has $3n+3-p$ edges and $n+2-2g$ vertices, so the average degree of a vertex is
\[ \frac{2(3n+3-p)}{n+2-2g} \leq \frac{6n}{2\eps n}=\frac{3}{\eps}.\]
Therefore, at least half of the vertices have degree at most $\frac{6}{\eps}$. There are $n+2-2g \geq 2\eps n$ vertices in $t$, so at least $\eps n$ of them have degree not greater than $\frac{6}{\eps}$.

Let $v^1, \dots, v^{\eps n}$ be such vertices and, for every $1 \leq i \leq \eps n$, let $f^i$ be a face incident to $v^i$.
Since each face is incident to only $3$ vertices, the set $F=\{f^1, \dots, f^{\eps n}\}$ contains at least $\frac{\eps}{3}n$ faces. It is sufficient to find $\frac{\eps}{12}n$ pairs $(f,f') \in F^2$ with $d^*(f,f') \leq \frac{24}{\eps}$. This will follow from the fact that balls (for $d^*$) centered at the elements of $F$ must strongly overlap. More precisely, let $r=\frac{12}{\eps}$. Since the dual map $t^*$ is connected, for any $f \in F$, we have\footnote{Unless the number of faces of $t$ is smaller than $r$, in which case $B_r^*(f)=t^*$, so any pair of $F^2$ is good.} $|B_r^*(f)| \geq r$. Therefore, we have
\[\sum_{f \in F} \left| B^*_{r}(f) \right| \geq \frac{12}{\eps} \times \frac{\eps}{3}n = 4n,\]
whereas $|t^*|=2n-p+2<4n$. Hence, there must be an intersection between the balls $B^*_{r}(f)$, so there are $f_1, f'_1 \in F$ such that $d^*(f_1,f'_1) \leq 2r = \frac{24}{\eps}$, and the pair $(f_1, f'_1)$ is good. We set $F_1=F \backslash \{f_1\}$. We now try to find a good pair in $F_1$ and remove an element of this pair, and so on. Assume that $F_i$ is the set $F$ where $i$ elements have been removed. If $i<\frac{\eps}{12} n$, then we have
\[\sum_{f \in F_i} \left| B^*_{r}(f) \right| \geq \frac{12}{\eps} \left( \frac{\eps}{3} n-i \right) \geq 3n> |t^*|, \]
so $F_i$ contains a good pair. Therefore, the process will not stop before $i=\frac{\eps}{12}n$, so we can find $\frac{\eps}{12} n$ good pairs in $F^2$, which concludes the proof.
\end{proof}

We are now able to prove the bounded ratio lemma.

\begin{proof}[Proof of Lemma \ref{ratio_lem}]
Let $g \geq 0, n \geq 2$ be such that $\frac{g}{n}\leq\frac{1}{2}-\eps$. We will define an "almost-injection" $\Phi$ from $\T_p(n,g)$ to $\T_p(n-1,g)$. The input will be a triangulation $t \in\mathcal{T}_p(n,g)$ with a marked good pair $(v_{1},v_{2})$. By Lemma \ref{lem_good_pairs}, the number of inputs is at least
\begin{equation}\label{injection_number_inputs}
\frac{\eps}{12} n \tau_p(n,g).
\end{equation}

Given an input $(t, v_1, v_2)$, let $(f_{1},f_{2},\ldots,f_{j})$ be the shortest path in $t^*$ from $v_{1}$ to $v_{2}$. Since the pair $(v_1,v_2)$ is good, we have $j \leq \frac{24}{\eps}$. For all $i$, let $e_{i}$ be the edge separating $f_{i}$ from $f_{i+1}$. We flip $e_{1}$, then $e_{2}$ and so on up to $e_{j-1}$ (see Figure \ref{fig_bdd_ratio}). Note that these flips are always well defined, since the faces $f_i$ are pairwise distinct.

As we do so, we keep track of all the flipped edges and the order in which they come. All the edges that were flipped are now incident to $v_{1}$, and the last of them is also incident to $v_{2}$. We then contract this edge and merge $v_1$ and $v_2$ into a vertex $v$, which creates two digons incident to $v$, that we contract into two edges. Finally, we mark the vertex $v$ obtained by merging $v_1$ with $v_2$, and we also mark the two edges obtained by contracting the digons.

\begin{figure}
\center\includegraphics[scale=0.7]{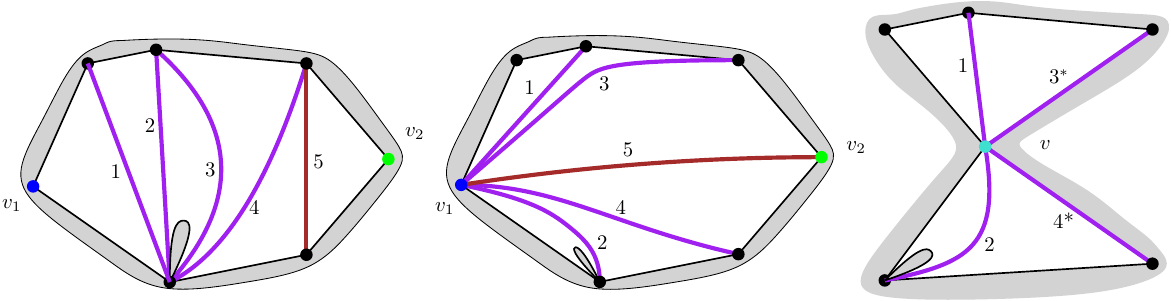}\caption{The injective mapping. On the left, a good pair and a path of triangles. In the center, the triangulation after the flips (we flipped the edges $1, 2, 3, 4, 5$ in this order). On the right, the final map, after contraction of the brown edge. The stars indicate the contracted digons.}\label{fig_bdd_ratio}
\end{figure}

These operations do not change the genus and the boundary length, and remove exactly $1$ vertex. Hence, the output of $\Phi$ is a triangulation of $\T_p(n-1,g)$, with a marked vertex $v$ of degree at most $\frac{36}{\eps}$ (since $\deg(v)<\deg(v_1)+\deg(v_2)+j$), two marked edges incident to $v$ and an ordered list of edges incident to $v$. Moreover, given the triangulation $\Phi(t)$, there are at most $n+2-2g \leq n+2$ possible values of $v$. Since $\deg(v) \leq \frac{36}{\eps}$, once $v$ is fixed, there are at most $\left( \frac{36}{\eps} \right)^2$ ways to choose the two marked edges and $\left( \frac{36}{\eps} \right)^{24/\eps}$ ways to choose the ordered list of edges. Hence, the number of possible outputs of $\Phi$ is at most
\begin{equation}\label{injection_number_outputs}
n \left( \frac{36}{\eps} \right)^{24/\eps+2} \tau_p(n-1,g).
\end{equation}
Finally, it is easy to see that $\Phi$ is injective: to go backwards, one just needs to duplicate the two marked edges, split $v$ in two between the two digons, and flip back the edges in the prescribed order. Therefore, by \eqref{injection_number_inputs} and \eqref{injection_number_outputs}, we obtain
\[ \frac{\eps}{12} n \tau_p(n,g) \leq n \left( \frac{36}{\eps} \right)^{24/\eps+2} \tau_p(n-1,g),\]
which concludes the proof with $C_{\eps}=\left( \frac{36}{\eps} \right)^{24/\eps+3}$.
\end{proof}

\subsection{Planarity and one-endedness}
\label{subsec_planar_one}

We now fix a sequence $(g_n)$ with $\frac{g_n}{n} \to \theta \in \left[0, \frac{1}{2} \right)$. As explained in Section \ref{subsec_local_topology}, the tightness of $(T_{n,g_n})$ for $d_{\loc}^*$ is immediate. Throughout this section, we will denote by $T$ a subsequential limit in distribution. It must be an infinite triangulation. We will first prove that $T$ is planar and one-ended, and then that its vertices have finite degrees. To establish planarity, the idea will be to bound, for any non-planar finite triangulation $t$, the probability that $t \subset T_{n,g_n}$ for $n$ large. For this, we will need the following combinatorial estimate.

\begin{lem}\label{lem_calcul_planarite}
Fix $k \geq 1$ and $m \in \Z$, numbers $\ell_{1}, \dots \ell_k \geq 1$ and perimeters $p_i^j \geq 1$ for $1 \leq j \leq k$ and $1 \leq i \leq \ell_j$. Then
\begin{equation}\label{eqn_lem_calcul_planarite}
\sum_{\substack{n_1+\dots+n_k=n+m \\ h_1+\dots+h_k=g_n-1-\sum_{j} (\ell_j-1)}} \prod_{j=1}^k \tau_{p^j_1, \dots, p^j_{\ell_j}}(n_j, h_j) = o \left( \tau(n,g_n) \right)
\end{equation}
when $n \to +\infty$.
\end{lem}

\begin{proof}
By Lemma \ref{lem_fill_holes}, the left-hand side of \eqref{eqn_lem_calcul_planarite} can be bounded by
\[ C \sum_{\substack{n_1+\dots+n_k=n+m \\ h_1+\dots+h_k=g_n-1-\sum_{j} (\ell_j-1)}} \prod_{j=1}^k n_j^{\ell_j-1} \tau(n_j, h_j) \]
for $C=6^{\sum(\ell_j-1)}$. We now use the crude bound $\tau(n_j, h_j) \leq f(n_j,h_j)$, where the numbers $f(n,g)=(3n+2)\tau(n,g)$ are those that appear in the Goulden--Jackson formula \eqref{eqn_Goulden_Jackson}. Since $\ell_j \geq 1$, we can bound $n_j^{\ell_j-1}$ by $2 n^{\ell_j-1}$ for $n\geq m$. The left-hand side of \eqref{eqn_lem_calcul_planarite} is then bounded by
\[C n^{\sum_{j=1}^k (\ell_j-1)} \sum_{\substack{n_1+\dots+n_k=n+m \\ h_1+\dots+h_k=g_n-1-\sum_{j} (\ell_j-1)}} \prod_{j=1}^k f(n_j, h_j).\]
Moreover, the Goulden--Jackson formula implies that
\begin{equation}\label{eqn_GJ_consequence_1}
\sum_{\substack{n_1+n_2=n \\ h_1+h_2=g}} f(n_1, h_1) f(n_2, h_2) \leq f(n+2,g)
\end{equation}
for any $n,g \geq 0$. By an easy induction on $k$, we obtain
\[\sum_{\substack{n_1+\dots+n_k=n \\ h_1+\dots+h_k=g}} \prod_{j=1}^k f(n_j, h_j) \leq f(n+2k-2,g).\]
Therefore, we can bound the left-hand side of \eqref{eqn_lem_calcul_planarite} by
\begin{equation}\label{gros_calcul_intermediaire}
C n^{\sum_{j=1}^k (\ell_j-1)} f \left( n+m+2k-2, g_n-1-\sum_{j=1}^k (\ell_j-1) \right).
\end{equation}
The Goulden--Jackson formula implies that 
\begin{equation}\label{eqn_GJ_consequence_2}
f(n-2,g-1) \leq n^{-2} f(n,g)
\end{equation}
for any $n$ and $g$, so $f(n,g-i) \leq n^{-2i} f(n+2i,g)$ for any $n$ and $1 \leq i \leq g$. Therefore, from \eqref{gros_calcul_intermediaire}, we obtain the asymptotic bound
\[C n^{\sum_{j=1}^k (\ell_j-1)} n^{-2-2\sum_{j=1}^k (\ell_j-1)} f \left( n+m+2k+2\sum (\ell_j-1) ,g_n \right) \leq C' n^{-2 -\sum_{j} (\ell_j-1)} f(n,g_n),\]
where in the end we use Lemma \ref{ratio_lem} (which results in a change in the constant). In particular, the left-hand side of \eqref{eqn_lem_calcul_planarite} is $o \left( \frac{f(n,g_n)}{n} \right)$, so it is $o \left( \tau(n,g_n) \right)$.
\end{proof}

\begin{corr}\label{conclusion_planarite}
Every subsequential limit of $\left( T_{n,g_n} \right)$ for $d_{\loc}^*$ is a.s. planar.
\end{corr}

\begin{proof}
If a subsequential limit $T$ is not planar, then we can find a finite triangulation $t$ with holes and with genus $1$ such that $t \subset T$. Indeed, if we explore $T$ triangle by triangle, the genus may only increase by at most $1$ at each step, so if the genus is positive at some point during the exploration, it must be $1$ at some point. Therefore, it is enough to prove that for any such triangulation $t$, we have
\[\P \left( t \subset T_{n, g_n} \right) \xrightarrow[n \to +\infty]{} 0.\]

If $t \subset T_{n,g_n}$, let $T^1, \dots, T^k$ be the connected components of $T_{n,g_n} \backslash t$. These components define a partition of the set of holes of $t$, where a hole $h$ is in the $j$-th class if $T^j$ is the connected component glued to $h$ (for example, on Figure \ref{fig_t_subset_T}, the three classes have sizes $2$, $1$ and $1$).
Note that the number of possible partitions is finite and depends only on $t$ (and not on $n$). Therefore, it is enough to prove that for any partition $\pi$ of the set of holes of $t$, we have
\begin{equation}\label{eqn_planarite_partition}
\P \left( \mbox{$t \subset T_{n, g_n}$ and the partition defined by $T_{n,g_n}$ is $\pi$} \right) \xrightarrow[n \to +\infty]{} 0.
\end{equation}
If this occurs, for each $j$, let $\ell_j$ be the number of holes of $T$ glued to $T^j$ and let $p^j_1, \dots, p^j_{\ell_j}$ be the perimeters of these holes. Then the connected component $T^j$ is a triangulation of the $(p^j_1, \dots, p^j_{\ell_j})$-gon (see Figure \ref{fig_t_subset_T}). Moreover, if $T_j$ has genus $h_j$, then the total genus of $T_{n, g_n}$ is equal to
\[ 1+\sum_{j=1}^k h_j + \sum_{j=1}^k (\ell_j-1),\]
so this sum must be equal to $g_n$, so
\[ \sum_{j=1}^k h_j =g_n-1-\sum_{j=1}^k (\ell_j-1). \]
Moreover, let $n_j$ be such that $T^j$ belongs to $\T_{p^j_1, \dots, p^j_{\ell_j}}(n_j,h_j)$. An easy computation shows that
\[ \sum_{j=1}^k n_j=n+m\]
with
\[ m=\frac{1}{2} \left( -|F(t)|+\sum_{j=1}^k \sum_{i=1}^{\ell_j} (p_i^j-2) \right) \in \Z,\]
where $|F(t)|$ is the number of triangles of $t$.

Therefore, the number of triangulations $T \in \T(n,g_n)$ such that $t \subset T$ and the resulting partition of the holes is equal to $\pi$ is the number of ways to choose, for each $j$, a triangulation of the $(p^j_1, \dots, p^j_{\ell_j})$-gon, such that the total genus of these triangulations is $g_n-1-\sum_{j=1}^k (\ell_j-1)$, and their total size is $n+m$. This is equal to the left-hand side of Lemma \ref{lem_calcul_planarite}, so \eqref{eqn_planarite_partition} is a consequence of Lemma \ref{lem_calcul_planarite}, which concludes the proof.
\end{proof}

The proof of one-endedness will be similar, but the combinatorial estimate that is needed is slightly different.

\begin{lem}\label{lem_calcul_oneended}
\begin{itemize}
\item
Fix $k \geq 1$, $m \in \Z$, numbers $\ell_{1}, \dots \ell_k \geq 1$ that are not all equal to $1$, and perimeters $p_i^j \geq 1$ for $1 \leq j \leq k$ and $1 \leq i \leq \ell_j$. Then
\begin{equation}\label{eqn_oneended_1}
\sum_{\substack{n_1+\dots+n_k=n+m \\ h_1+\dots+h_k=g_n-\sum_{j} (\ell_j-1)}} \prod_{j=1}^k \tau_{p_1^j, \dots, p_{\ell_j}^j}(n_j, h_j) = o \left( \tau(n,g_n) \right).
\end{equation} 
\item
Fix $k \geq 2$, $m \in \Z$ and perimeters $p_1, \dots, p_k$. There is a constant $C$ such that, for every $a$ and $n$, we have
\begin{equation}\label{eqn_oneended_2}
\sum_{\substack{n_1+\dots+n_k=n+m \\ h_1+\dots+h_k=g_n \\ n_1, n_2 >a}} \prod_{j=1}^k \tau_{p_j}(n_j, h_j) \leq \frac{C}{a} \tau(n,g_n).
\end{equation}
\end{itemize}
\end{lem}

\begin{proof}
We start with the first point. The proof is very similar to the proof of Lemma \ref{lem_calcul_planarite}, with the following difference: here, the sum of the genuses differs by one, so we will lose a factor $n^2$ in the end of the computation. This forces us to be more careful in the beginning and to use the assumption that the $\ell_j$ are not all equal to $1$.

More precisely, by using Lemma \ref{lem_fill_holes} and the bound $\tau(n,g) \leq \frac{1}{n} f(n,g)$, as in the proof of Lemma \ref{lem_calcul_planarite}, the left-hand side of \eqref{eqn_oneended_1} can be bounded by
\[ C \sum_{\substack{n_1+\dots+n_k=n+m \\ h_1+\dots+h_k=g_n-\sum_{j} (\ell_j-1)}}  \left( \prod_{j=1}^k n_j^{\ell_j-2} \right) \left( \prod_{j=1}^k f(n_j,h_j) \right),\]
where $C$ does not depend on $n$.
Without loss of generality, assume that $\ell_1 \geq 2$. Then we have $n_1^{\ell_1-2} \leq (n+m)^{\ell_1-2}$. Moreover, for every $j \geq 2$, we have $n_j^{\ell_j-2} \leq n_j^{\ell_j-1} \leq (n+m)^{\ell_j-1}$ since $\ell_j \geq 1$. Therefore, we obtain for $n \geq m$:
\[ \prod_{j=1}^k n_j^{\ell_j-2} \leq 2 n^{\sum_j (\ell_j-1)-1}.\]
By using this and the Goulden--Jackson formula in the same way as in the proof of Lemma \ref{lem_calcul_planarite}, we obtain the bound
\[ C n^{-1-\sum_{j} (\ell_j-1)} f \left( n+m+2k-2, g_n \right) \leq 
C' n^{-1-\sum_{j} (\ell_j-1)} f(n,g_n),\]
where the last inequality follows from the bounded ratio lemma. Since $\ell_1 \geq 2$, we have $\sum_j (\ell_j-1) \geq 1$, so this is $o \left( \frac{f(n,g_n)}{n} \right)$ and we get the result.

We now prove the second point. As in the first case (but with $\ell_j=1$ for every $j$), the left-hand side can be bounded by
\[C \sum_{\substack{n_1+\dots+n_k=n+m \\ h_1+\dots+h_k=g_n \\ n_1, n_2 >a}} \left( \prod_{j=1}^k \frac{1}{n_j} \right) \left( \prod_{j=1}^k f(n_j, h_j) \right).\]

Moreover, if $n_1, n_2>a$, then at least one of the $n_j$ is larger than $\frac{n+m}{k}$ and two are larger than $a$, so $\prod_{j=1}^k n_j \geq \frac{(n+m)a}{k}$, so we obtain the bound (for $n$ large enough, with $C'$ and $C''$ independent of $n$ and $a$)
\begin{align*}
\frac{C'}{an} \sum_{\substack{n_1+\dots+n_k=n+m \\ h_1+\dots+h_k=g_n \\ n_1, n_2 > a}} \prod_{j=1}^k f(n_j, h_j) & \leq \frac{C'}{an} \sum_{\substack{n_1+\dots+n_k=n+m \\ h_1+\dots+h_k=g_n}} \prod_{j=1}^k f(n_j, h_j)\\
& \leq \frac{C'}{an} f(n+m+2k-2,g_n)\\
& \leq \frac{C''}{an} f(n,g_n)\\
& \leq \frac{C''}{a} \tau(n,g_n),
\end{align*}
where we use the Goulden--Jackson formula to reduce the sum and finally the bounded ratio lemma, in the same way as previously.
\end{proof}

\begin{corr}\label{conclusion_oneended}
Every subsequential limit $T$ of $\left( T_{n,g_n} \right)$ for $d_{\loc}^*$ is a.s. one-ended in the sense that, for every finite triangulation $t$ with holes such that $t \subset T$, only one hole of $t$ is filled with infinitely many faces\footnote{This is a "weak" definition of one-endedness, since it does not prevent $T$ to be the dual of a tree. However, once we will have proved that $T$ has finite vertex degrees, this will be equivalent to the usual definition.}.
\end{corr}

\begin{proof}
The proof is quite similar to the proof of Corollary \ref{conclusion_planarite}, but with Lemma \ref{lem_calcul_oneended} playing the role of Lemma \ref{lem_calcul_planarite}.

More precisely, if a subsequential limit $T$ is not one-ended with positive probability, it contains a finite triangulation $t$ such that two of the connected components of $T \backslash t$ are infinite. This means that we can find $\eps>0$, a triangulation $t$ and two holes $h_1, h_2$ of $t$ such that, for every $a>0$,
\begin{equation}\label{eqn_oneended_t}
P \left( \mbox{$t \subset T$ and $T \backslash t$ has two connected components with at least $a$ faces} \right) \geq \eps.
\end{equation}
By Corollary \ref{conclusion_planarite}, we can assume that $t$ is planar. 
If this holds, then $T$ contains a finite triangulation obtained by starting from $t$ and adding $a$ faces in the hole $h_1$ and $a$ faces in the hole $h_2$. We denote by $t^{a,a}$ the set of such triangulations. Then \eqref{eqn_oneended_t} means that for any $a>0$, for $n$ large enough, we have
\begin{equation}\label{eqn_oneended_finite}
\P \left( \mbox{$T_{n, g_n}$ contains a triangulation of $t^{a,a}$} \right) \geq \eps.
\end{equation} 
This can occur in two different ways, which will correspond to the two items of Lemma \ref{lem_calcul_oneended}:
\begin{itemize}
\item[(i)]
either at least one connected component of $T_{n, g_n} \backslash t$ is glued to at least two holes of $t$,
\item[(ii)]
or the $k$ holes of $t$ correspond to $k$ connected components $T^1, \dots, T^k$, where $T^1$ and $T^2$ have size at least $a$.
\end{itemize}
In case (i), the connected components of $T_{n,g_n}$ are triangulations of multi-polygons, at least one of which has two boundaries. The proof that the probability of this case goes to $0$ is now the same as the proof of Corollary \ref{conclusion_planarite}, but we use the first point of Lemma \ref{lem_calcul_oneended}. Note that the assumption the $\ell_j$ are not all $1$ comes from the fact that one of the connected components is glued to two holes. Moreover, the sum of the genuses of the $T^j$ is $g-\sum_j (\ell_j-1)$ and not $g-1-\sum_j (\ell_j-1)$ because this time $t$ has genus $0$ and not $1$.

Similarly, in case (ii), the $k$ holes of $t$ must be filled with $k$ triangulations of a single polygon, two of which have at least $a$ faces, so they belong to a set of the form $\T_{p_j}(n_j,h_j)$ with $n_j \geq \frac{a}{2}$ if $a$ is large enough compared to the perimeters of the holes. Hence, the second point of Lemma \ref{lem_calcul_oneended} allows to bound the number of ways to fill these holes. We obtain that, for $a$ large enough, we have
\[ \P \left( \mbox{$T_{n, g_n}$ contains a triangulation of $t^{a,a}$} \right) \leq o(1)+\frac{2C}{a} \]
as $n \to +\infty$, where $o(1)$ comes from case (i) and $\frac{2C}{a}$ from case (ii). This contradicts \eqref{eqn_oneended_finite}, so $T$ is a.s. one-ended.
\end{proof}

\subsection{Finiteness of the degrees}

Our goal is now to prove tightness for $d_{\loc}$. As before, let $(g_n)$ be a sequence with $\frac{g_n}{n} \to \theta \in \left[ 0,\frac{1}{2} \right)$.

\begin{prop}\label{prop_tightness_dloc}
The sequence $(T_{n,g_n})$ is tight for $d_{\loc}$.
\end{prop}

Let $T$ be a subsequential limit of $(T_{n,g_n})$ for $d_{\loc}^*$. By Lemma \ref{lem_dual_convergence}, to finish the proof of tightness for $d_{\loc}$, we only need to show that almost surely, all the vertices of $T$ have finite degrees. As in \cite{AS03}, we will first study the degree of the root vertex, and then extend finiteness by using invariance under the simple random walk. The main difference with \cite{AS03} is that, while \cite{AS03} uses exact enumeration results, we will rely on the bounded ratio lemma.

\begin{lem}\label{lem_root_degree_is_finite}
The root vertex of $T$ has a.s. finite degree.
\end{lem}

\begin{proof}
We follow the approach of \cite{AS03} and perform a filled-in peeling exploration of $T$. Before expliciting the peeling algorithm that we use, note that we already know by Corollary \ref{conclusion_planarite} that the explored part will always be planar, so no peeling step will merge two different existing holes. Moreover, by Corollary \ref{conclusion_oneended}, if a peeling step separates the boundary into two holes, then one of them has finitely many faces inside, so it will be filled with a finite triangulation. Therefore, at each step, the explored part will be a triangulation with a single hole.

The peeling algorithm $\mathcal{A}$ that we use is the following: if the root vertex $\rho$ belongs to $\partial t$, then $\mathcal{A}(t)$ is the edge on $\partial t$ on the left of $\rho$. If $\rho \notin \partial t$, then the exploration is stopped. Since only finitely many edges incident to $\rho$ are added at each step, it is enough to prove that the exploration will a.s. eventually stop. We recall that $\expl_T^{\mathcal{A}}(i)$ is the explored part at time $i$.

We will prove that at each step, conditionally on $\expl_T^{\mathcal{A}}(i)$, the probability to swallow the root and finish the exploration at time $i+1$ or $i+2$ is bounded from below by a positive constant. For every triangulation $t$ with one hole such that $\rho \in \partial t$, we denote by $t^+$ the triangulation constructed from $t$ as follows (see Figure \ref{fig_swallowing_root}):
\begin{itemize}
\item
we first glue a triangle to the edge of $\partial t$ on the left of $\rho$, in such a way that the third vertex of this triangle does not belong to $t$, to obtain a triangulation with perimeter at least $2$;
\item
we then glue a second triangle to the two edges of the boundary incident to $\rho$.
\end{itemize}

\begin{figure}
\centering
\includegraphics[scale=0.5]{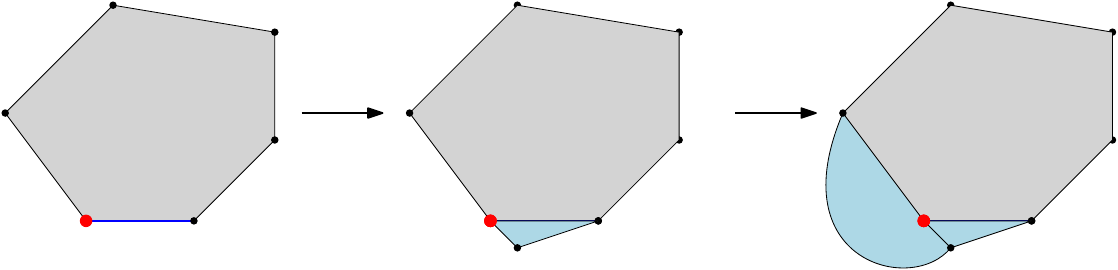}\hspace{2cm}\includegraphics[scale=0.5]{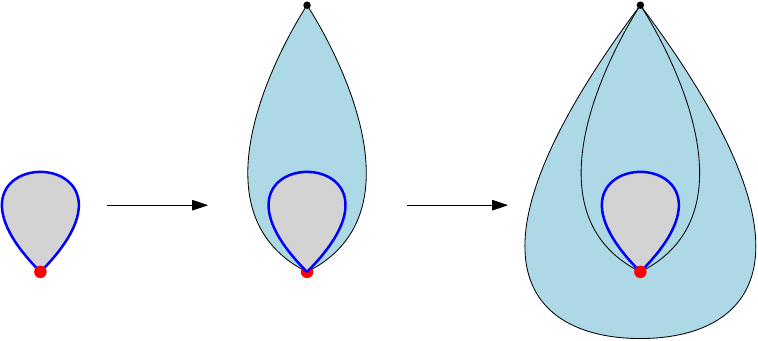}
\caption{The construction of $t^+$ from $t$. In gray, the triangulation $t$. In red, the root vertex. In blue, the new triangles. In the bottom, the case $|\partial t|=1$.}
\label{fig_swallowing_root}
\end{figure}

Note that $t^+$ is a planar map with the same perimeter as $t$ but one more vertex and two more triangles. By the choice of our peeling algorithm, if $\expl_T^{\mathcal{A}}(i)^+ \subset T$, then we have $\expl_T^{\mathcal{A}}(i+2)=\expl_T^{\mathcal{A}}(i)^+$. Moreover, if this is the case, the exploration is stopped at time $i+2$. Hence, it is enough to prove that the quantity
\[ \P \left( t^+ \subset T | t \subset T \right)\]
is bounded from below for finite, planar triangulations $t$ with a single hole and $\rho \in \partial t$.

We fix such a $t$, with perimeter $p$. Let also $(n_k)$ be a sequence of indices such that $T_{n_k,g_{n_k}}$ converges in distribution to $T$. We have
\[ \P \left( t^+ \subset T | t \subset T \right)=\lim_{k \to +\infty} \frac{\P \left( t^+ \in T_{n_k,g_{n_k}} \right)}{\P \left( t \in T_{n_k,g_{n_k}} \right)}  = \lim_{k \to +\infty} \frac{\tau_p \left( n_k+m-1, g_{n_k} \right)}{\tau_p(n_k+m,g_{n_k})}, \]
where $m=\frac{p-2-|F(t)|}{2}$ and $|F(t)|$ is the number of triangles of $t$. Moreover, there is $\eps>0$ such that $g_n \leq \left( \frac{1}{2} - 2\eps \right) n$ for $n$ large enough, so $\frac{g_{n_k}}{n_k+m} \leq \frac{1}{2}-\eps$ for $k$ large enough. By the bounded ratio lemma, we obtain
\[\P \left( t^+ \subset T | t \subset T \right) \geq \frac{1}{C_{\eps}}\]
for every $t$, which concludes the proof.
\end{proof}

\begin{rem}
Our proof shows that the number of steps needed to swallow the root has exponential tail. However, since we do not control the finite triangulations filling the holes that may appear, it does not give any quantitative bound on the root degree.
\end{rem}

\begin{proof}[Proof of Proposition \ref{prop_tightness_dloc}]
Let $T$ be a subsequential limit of $(T_{n,g_n})$ for $d_{\loc}^*$. By Lemma \ref{lem_dual_convergence}, to obtain tightness, it is enough to prove that almost surely, all the vertices of $T$ have finite degrees. The argument is essentially the same as in \cite{AS03} and relies on Lemma \ref{lem_root_degree_is_finite} and invariance under the simple random walk.

More precisely, for every $n$, let $\overrightarrow{e}_0^n$ be the root edge of $T_{n,g_n}$ and let $\overrightarrow{e}_0$ be the root of $T$. We first note that the distribution of $T_{n,g_n}$ is invariant under reversing the orientation of the root, so this is also the case of $T$. By Lemma \ref{lem_root_degree_is_finite}, this implies that the endpoint of $\overrightarrow{e}_0$ has a.s. finite degree.

We then denote by $\overrightarrow{e}_1^n$ the first step of the simple random walk on $T_{n,g_n}$: its starting point is the endpoint of $\overrightarrow{e}_0^n$ and its endpoint is picked uniformly among all the neighbours of the starting point. Since the endpoint of $\overrightarrow{e}_0$ has finite degree, we can also define the first step $\overrightarrow{e}_1$ of the simple random walk on $T$.
For the same reason as in the planar case (see Theorem 3.2 of \cite{AS03}), the triangulations $(T_{n,g_n}, \overrightarrow{e}_1^n)$ and $(T_{n,g_n}, \overrightarrow{e}_0^n)$ have the same distribution, so $(T,\overrightarrow{e}_1)$ has the same distribution as $(T, \overrightarrow{e}_0)$. In particular, all the neighbours of the endpoint of $\overrightarrow{e}_0$ must have finite degrees. From here, an easy induction on $i$ shows that for every $i \geq 0$, we can define the $i$-th step $\overrightarrow{e}_i$ of the simple random walk on $T$, that $(T, \overrightarrow{e}_i)$ has the same distribution as $(T, \overrightarrow{e}_0)$ and that all vertices at distance $i$ from the root in $T$ are finite. This proves the tightness for $d_{\loc}$.

Planarity is proved by Corollary~\ref{conclusion_planarite}. Finally, it easy to check that for triangulations with finite vertex degrees, the weak version of one-endedness proved in Corollary~\ref{conclusion_oneended} implies the usual one. For example, if $V$ is a finite set of vertices of $T$, one can consider a finite, connected triangulation $t \subset T$ containing all the faces and edges incident to vertices of $V$.
\end{proof}

\section{Weakly Markovian triangulations}

The goal of this section is to prove Theorem \ref{thm_weak_markov}. We first recall the definition of a weakly Markovian triangulation.

\begin{defn}
Let $T$ be a random infinite triangulation of the plane. We say that $T$ is \emph{weakly Markovian} if there is a family $(a^p_v)_{v \geq p \geq 1}$ of numbers with the following property: for every triangulation $t$ with a hole of perimeter $p$ and $v$ vertices in total, we have
\[\P \left( t \subset T \right)=a^p_v.\]
\end{defn}

By their definition, the PSHT $\TT_{\lambda}$ are weakly Markovian. We denote by $a^p_v(\lambda)=C_p(\lambda) \lambda^v$ the associated constants, where $C_p(\lambda)$ is given by \eqref{eqn_cp_psht}. This implies that any mixture of these is also weakly Markovian. Indeed, for any random variable $\Lambda$ with values in $(0,\lambda_c]$, we denote by $\TT_{\Lambda}$ the PSHT with random parameter $\Lambda$. Let also $\mu$ be the distribution of $\Lambda$. Then for every triangulation $t$ with a hole of perimeter $p$ and $v$ vertices in total, we have
\begin{equation}\label{defn_apv_mu}
\P \left( t \subset \TT_{\Lambda} \right) = \int_0^{\lambda_c} \P \left( t \subset \TT_{\lambda} \right) \mu(\mathrm{d}\lambda) = \int_0^{\lambda_c} C_p(\lambda) \lambda^v \mu(\mathrm{d}\lambda) =: a^p_v[\mu].
\end{equation}
Note that the last integral always converges since $C_p(\lambda) \lambda^v$ is bounded by $1$. Therefore, the triangulation $\TT_{\Lambda}$ is weakly Markovian. Our goal here is to prove Theorem \ref{thm_weak_markov}, which states that the converse is true.

In the remainder of this section, we will denote by $T$ some weakly Markovian random triangulation, and by $(a^p_v)_{v \geq p \geq 1}$ the associated constants. Before giving an idea of the proof, let us start with a remark that will be very useful in all that follows. The numbers $a^p_v$ are linked to each other by linear equations that we call the \emph{peeling equations}. In this section, for every $p \geq 1$ and $j \geq 0$, we denote by $|\mathcal{T}_{p}(j)|$ the number of planar triangulations of a $p$-gon (rooted on the boundary) with exactly $j$ inner vertices\footnote{In order to have nicer formulas, the convention we use here differs from the rest of the paper, in which the parameter $n$ is related to the \emph{total} number of vertices. This is why we do not use the $\tau$ notation. We insist that this holds only in Section 3.}. We also adopt the convention that $|\mathcal{T}_{2}(0)|=1$ (i.e. a hole of perimeter $2$ can be filled by simply gluing the two edges to each other). On the other hand $|\mathcal{T}_{1}(0)|=0$.

\begin{lem}\label{lem_peeling_equation}
For every $v \geq p \geq 1$, we have
\begin{equation}\label{peeling_equation}
a^p_v=a^{p+1}_{v+1}+2\sum_{i=0}^{p-1} \sum_{j=0}^{+\infty} |\mathcal{T}_{i+1}(j)| a^{p-i}_{v+j}.
\end{equation}
\end{lem}

In particular, the sum in the right-hand side must converge. Note also that if $v \geq p \geq 1$, then $v+j \geq p-i \geq 1$ in all the terms of the sum, so all the terms make sense.

\begin{proof}[Proof of Lemma \ref{lem_peeling_equation}]
The proof just consists of making one peeling step. Assume that $t \subset T$ for some triangulation $t$ with a hole of perimeter $p$ and volume $v$. Fix an edge $e \in \partial T$, and consider the face $f$ out of $t$ that is adjacent to $e$. Then we are in exactly one of the three following cases:
\begin{itemize}
\item
the third vertex of $f$ does not belong to $\partial t$, 
\item
the third vertex of $f$ belongs to $\partial t$, and $f$ separates on its left $i$ edges of $\partial t$ from infinity,
\item
the third vertex of $f$ belongs to $\partial t$, and $f$ separates on its right $i$ edges of $\partial t$ from infinity.
\end{itemize}
In the first case, the triangulation we obtain by adding $f$ to $t$ has perimeter $p+1$ and $v+1$ vertices in total. In the other two cases, $f$ separates $T \backslash t$ in two components, one finite and one infinite. The finite component has perimeter $i+1$. If it has $j$ inner vertices, then after filling the finite component, we obtain a triangulation with perimeter $p-i$ and volume $v+j$.
\end{proof}

Our main task will be to prove the following result: it states that the numbers $a_v^1$ for $v \geq 1$ are compatible with some mixture of PSHT.

\begin{prop}\label{prop_perimeter_one}
There is a probability measure $\mu$ on $(0,\lambda_c]$ such that, for every $v \geq 1$, we have
\[a^1_v=a^1_v[\mu].\]
\end{prop}

Once Proposition \ref{prop_perimeter_one} is proved, Theorem~\ref{thm_weak_markov} easily follows. Indeed, Lemma \ref{lem_peeling_equation} can be rewritten
\begin{equation}\label{peeling_equation_rewritten}
a^{p+1}_{v+1}=a^p_v-2\sum_{i=0}^{p-1} \sum_{j=0}^{+\infty} |\mathcal{T}_{i+1}(j)| a^{p-i}_{v+j}
\end{equation}
for $v+1 \geq p+1 \geq 2$. Hence, numbers of the form $a^{p+1}_{v}$ can be expressed in terms of the numbers $a^{i}_{v}$ with $i \leq p$. On the other hand, the numbers $a^p_v[\mu]$ also satisfy~\eqref{peeling_equation_rewritten}. Therefore, by induction on $p$, we can prove that for every $v \geq p \geq 1$, we have
\[a^p_v=a^p_v[\mu].\]
Since the numbers $a^p_v$ characterize entirely the distribution of $T$, we are done.

Therefore, we only need to prove Proposition \ref{prop_perimeter_one}. As a particular case of \eqref{defn_apv_mu}, for every measure $\mu$ and any $v \geq 0$, we have
\[a^1_{v+1}[\mu]=\int_0^{\lambda_c} C_1(\lambda) \lambda^{v+1} \mu(\mathrm{d}\lambda) = \int_0^{\lambda_c} \lambda^{v} \mu(\mathrm{d}\lambda).\]
The right-hand side can be interpreted as the $v$-th moment of a random variable with distribution $\mu$. Therefore, all we need to prove is that there is a random variable $\Lambda$ with support in $(0,\lambda_c]$ such that, for every $v \geq 0$, we have
\[a^1_{v+1}=\E [\Lambda^v].\]

We first show that there exists such a variable with support in $[0,1]$. Since $a_1^1=1$, this is precisely the Hausdorff moment problem \cite{Hau21}, so all we need to show is that the sequence $(a^1_{v+1})_{v \geq 0}$ is completely monotonic. More precisely, let $\Delta$ be the discrete derivative operator:
\[(\Delta u)_n=u_n-u_{n+1}.\]
Then, to prove that $(a^1_{v+1})_{v \geq 0}$ is the moment sequence of a probability measure on $[0,1]$, it is sufficient to prove the following lemma.

\begin{lem}\label{lem_completely_monotonic}
For every $k \geq 0$ and $v \geq p \geq 1$, we have
\[ \left( \Delta^k a^p \right)_v \geq 0.\]
\end{lem}

Note that the case $p=1$ is sufficient. However, it will be more convenient to prove the lemma in the general case.

\begin{proof}
We prove the lemma by induction on $k$. The case $k=0$ just means that $a^p_v \geq 0$ for every $v \geq p \geq 1$. The case $k=1$ means that $a^p_v$ is non-decreasing in $v$, which is a straightforward consequence of the peeling equation: in the right-hand side of \eqref{peeling_equation}, the term corresponding to $i=0$ and $j=1$ is $|\mathcal{T}_{1}(1)| a^{p}_{v+1}$, so
\[a^p_v \geq 2 a^p_{v+1} \geq a^p_{v+1}.\]

Now assume that $k \geq 1$ and that the lemma is proved for $k$. We will use the result for $k$ and $p+1$ to prove it for $k+1$ and $p$. Let $v \geq p \geq 1$. By using the induction hypothesis and \eqref{peeling_equation_rewritten} for $(p,v), (p,v+1), \dots, (p,v+k)$, we obtain
\begin{align*}
0 & \leq \Delta^k (a^{p+1})_{v+1}\\
&= \sum_{\ell=0}^k (-1)^{\ell} \binom{k}{\ell} a^{p+1}_{v+1+\ell}\\
&= \sum_{\ell=0}^k (-1)^{\ell} \binom{k}{\ell} a^p_{v+\ell} -2\sum_{i=0}^{p-1} \sum_{j=0}^{+\infty} |\mathcal{T}_{i+1}(j)| \sum_{\ell=0}^k (-1)^{\ell} \binom{k}{\ell} a^{p-i}_{v+\ell+j}\\
&= \left( \Delta^k a^p \right)_v -2\sum_{i=0}^{p-1} \sum_{j=0}^{+\infty} |\mathcal{T}_{i+1}(j)| \left( \Delta^k a^{p-i} \right)_{v+j}.
\end{align*}
By the induction hypothesis, all the terms $\left( \Delta^k a^{p-i} \right)_{v+j}$ in the sum are nonnegative. Therefore, the above sum does not decrease if we remove some of the terms. In particular, we may remove all the terms except the one for which $i=0$ and $j=1$, and we may also remove the factor $2$. Since $|\mathcal{T}_{1}(1)|=1$, we obtain
\[0 \leq \left( \Delta^k a^p \right)_v - \left( \Delta^k a^{p} \right)_{v+1}= \left( \Delta^{k+1} a^p \right)_{v},\]
which proves the lemma by induction.
\end{proof}

\begin{rem}
Many bounds used in the last proof may seem very crude. This is due to the fact that we prove the existence of a variable $\Lambda$ with support in $[0,1]$, whereas its support is actually in $[0,\lambda_c]$. For example, if we had not removed the factors $2$, we would have obtained that $\Lambda$ has support in $\left[ 0, \frac{1}{2} \right]$ instead of $[0,1]$.
\end{rem}

\begin{proof}[End of the proof of Theorem \ref{thm_weak_markov}]
By Lemma \ref{lem_completely_monotonic}, there is a random variable $\Lambda$ with support in $[0,1]$ such that, for every $v \geq 1$, we have $a^1_{v+1}=\E [\Lambda^{v}]$. All we need to show is that $\Lambda \in (0,\lambda_c]$ almost surely. We first explain why $\Lambda \leq \lambda_c$. The peeling equation for $p=v=1$ shows that
\[\sum_{j=0}^{+\infty} |\mathcal{T}_{1}(j)| a^1_{j+1} \leq a^1_1 <+\infty. \]
On the other hand, we know that $|\mathcal{T}_{1}(j)| \sim c \lambda_c^{-j} j^{-5/2}$ as $j \to +\infty$ for some constant $c>0$. Therefore, if $\P \left( \Lambda \geq \lambda_c+\eps \right) \geq \eps$ for some $\eps>0$, then $a^1_{v+1} \geq \eps (\lambda_c+\eps)^v$ for every $v$ and the series above diverge, so we get a contradiction.

We finally prove that $\P \left( \Lambda=0 \right)=0$. As explained above, we already know that
\[a^p_v=\int_0^{\lambda_c} a^p_v(\lambda) \mu(\mathrm{d}\lambda), \]
where $\mu$ is the distribution of $\Lambda$. We also have $a^p_v(0)=\mathbbm{1}_{p=v}$. Therefore, if $\delta=\P \left( \Lambda=0 \right)$, for every $p \geq 1$, we have $a_p^p \geq \delta$. Now let $t_p$ be the triangulation with $p$ vertices and a hole of perimeter $p$ represented on Figure \ref{fig_trivial_triangulation}.
\begin{figure}
\begin{center}
\includegraphics[scale=0.5]{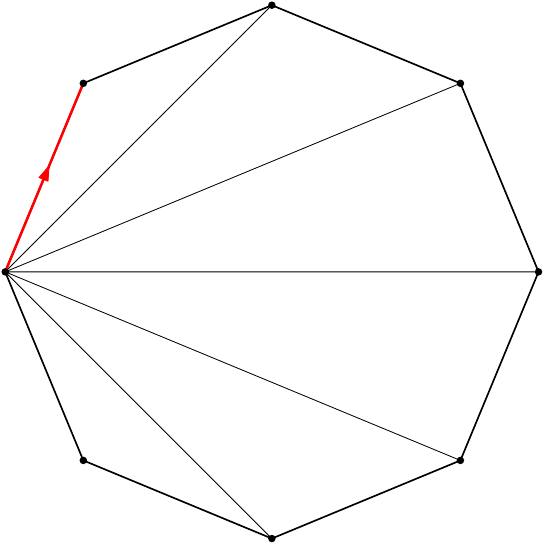}
\end{center}
\caption{The "triangle chain" triangulation $t_p$ with $p$ vertices (for $p=8$).}\label{fig_trivial_triangulation}
\end{figure}
For every $p \geq 1$, we have $\P \left( t_p \subset T \right)=a^p_p \geq \delta$. Since the events $t_p \subset T$ are nonincreasing in $p$, we have
\[ \P \left( \forall p \geq 1, t_p \subset T \right) \geq \delta.\]
On the other hand, if $t_p \subset T$, then the degree of the root vertex is at least $p+2$, so we have $\P \left( \deg(\rho)=+\infty \right) \geq \delta$. Since the degrees are finite, we must have $\delta=0$, so $\Lambda$ is supported in $(0,\lambda_c]$, which concludes the proof.
\end{proof}

\section{Ergodicity}\label{sec_ergodicity}

Throughout this section, we denote by $(g_n)$ a sequence such that $\frac{g_n}{n} \to \theta \in \left[ 0,\frac{1}{2} \right)$ and by $T$ a subsequential limit of $(T_{n, g_n})$ for $d_{\loc}$. In order to keep the notation light, we will always implicitly restrict ourselves to values of $n$ along which $T_{n, g_n}$ converges to $T$ in distribution.

For any $n$, the triangulation $T_{n,g_n}$ is weakly Markovian, so this is also the case of $T$. Therefore, by Theorem \ref{thm_weak_markov}, we know that $T$ must be a mixture of PSHT, i.e. there is a random variable $\Lambda \in (0,\lambda_c]$ such that $T$ has the same distribution as the PSHT with random parameter $\mathbb{T}_{\Lambda}$.

We also note right now that by the discussion in the end of Section \ref{subsec_psht} (Equation \eqref{eqn_asymptotic_peeling_psht}), the parameter $\Lambda$ is a measurable function of the triangulation $\TT_{\Lambda}$. More precisely, if $P^{\Lambda}$ and $V^{\Lambda}$ are the perimeter and volume processes associated with the peeling exploration of $\TT_{\Lambda}$, then $\Lambda$ can be defined as $f^{-1} \left( \lim_{i \to +\infty} \frac{P^{\Lambda}(i)}{V^{\Lambda}(i)} \right)$, where $f(\lambda)=1-4h$ is injective. In particular $\Lambda$ is defined without any ambiguity. Our goal is now to prove that $\Lambda$ is deterministic, and given by \eqref{eqn_lambda_vs_theta}.

\subsection{The two holes argument}\label{subsec_two_holes}

Roughly speaking, we know that $T_{n,g_n}$ seen from a typical point $e_n$ looks like a PSHT with random parameter $\Lambda$. We would like to show that $\Lambda$ does not depend on $\left( T_{n,g_n}, e_n \right)$. The first step is essentially to prove that $\Lambda$ does not depend on $e_n$.

More precisely, conditionally on $T_{n, g_n}$, we pick two independent uniform oriented edges $e^1_n, e^2_n$ of $T_{n, g_n}$. The pairs $(T_{n, g_n}, e^1_n)$ and $(T_{n, g_n}, e^2_n)$ have the same distribution, and both converge in distribution to $\mathbb{T}_{\Lambda}$. It follows that the pairs
\[\left( (T_{n, g_n}, e^1_n), (T_{n, g_n}, e^2_n) \right)\]
for $n \geq 1$ are tight, so up to further extraction, they converge to a pair $(\mathbb{T}^1_{\Lambda_1}, \TT^2_{\Lambda_2})$, where both marginals have the same distribution as $\TT_{\Lambda}$. By the above discussion, the variables $\Lambda_1$ and $\Lambda_2$ are well defined. By the Skorokhod representation theorem, we may assume the convergence in distribution is almost sure. Our goal in this subsection is to prove the following lemma.

\begin{prop}\label{prop_same_map_same_lambda}
Under the convergence
\[\left( (T_{n, g_n}, e^1_n), (T_{n, g_n}, e^2_n) \right) \xrightarrow[n \to +\infty]{a.s.} \left( \mathbb{T}^1_{\Lambda_1}, \TT^2_{\Lambda_2} \right), \]
we have $\Lambda_1=\Lambda_2$ almost surely.
\end{prop}

\begin{proof}
The idea of the proof is the following: if $\Lambda_1 \ne \Lambda_2$, consider two large neighbourhoods $N^1_n$ and $N^2_n$ (say, of size $100$) with the same perimeter around $e_n^1$ and $e_n^2$. Then there is a natural way to "swap" $N^1_n$ and $N^2_n$ in $T_{n,g_n}$ (cf. Figure \ref{fig_two_holes}), without changing its distribution. If we do this, then $T_{n,g_n}$ around $e^1_n$ looks like $\TT_{\Lambda_1}$ in a neighbourhood of the root of size $100$, and like $\TT_{\Lambda_2}$ out of this neighbourhood. However, if $\Lambda_1 \ne \Lambda_2$, such a configuration is highly unlikely for a mixture of PSHT.

\begin{figure}
\begin{center}
\includegraphics[scale=0.8]{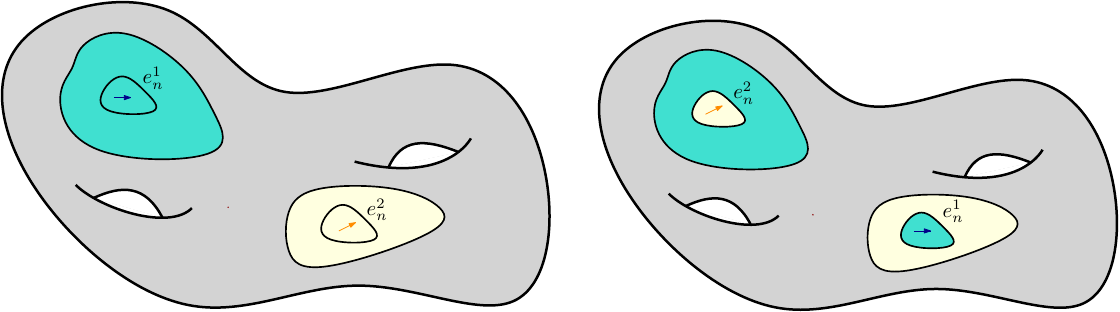}
\end{center}
\caption{The idea of the two-holes argument: the inner parts around $e_n^1$ and $e_n^2$ are "swapped". The root edges $e_n^1$ and $e_n^2$ are swapped at the same time as their neighbourhoods.}\label{fig_two_holes}
\end{figure}

More precisely, in the following proof, we consider a deterministic peeling algorithm $\mathcal{A}$, according to which we will explore $T_{n,g_n}$ around $e_n^1$ and around $e_n^2$. All the explorations we will consider will be filled-in: everytime the peeled face separates the undiscovered part in two, we discover completely the smaller part. While the computations in the beginning of the proof hold for any choice of $\mathcal{A}$, we will need later to specify the choice of the algorithm.

We denote by $\expl_n^1(i)$ the explored part at time $i$ when the exploration is started from $e_n^1$. If at some point $\expl_n^1(i)$ is non-planar, then the exploration is stopped. We write respectively $P_n^1(i)$ and $V_n^1(i)$ for its perimeter and its volume (i.e. its total number of vertices). 
For any $p \geq 2$, let $\tau_n^1(p)$ be the smallest $i$ such that  $P_n^1(i)=p$. We denote by $\expl_n^2(i)$, $P_n^2(i)$, $V_n^2(i)$ and $\tau_n^2(p)$ the analog quantities for the exploration started from $e_n^2$. We also denote by $\expl_{\infty}^1(i)$, $P_{\infty}^1(i)$, $V_{\infty}^1(i)$ and $\tau_{\infty}^1(p)$ the analog quantities for $\TT^1_{\Lambda_1}$, and similarly for $\TT^2_{\Lambda_2}$.

We fix $\eps>0$. We recall that
\[ \frac{P_{\infty}^1(i)}{V_{\infty}^1(i)} \xrightarrow[i \to +\infty]{a.s.} f(\Lambda_1), \]
where $f:(0,\lambda_c] \to [0,1)$ is bijective and decreasing.
Note also that the times $\tau_{\infty}^1(p)$ are a.s. finite and $\tau_{\infty}^1(p) \to +\infty$ when $p \to +\infty$, so for $p$ large enough, we have
\[ \P \left( \left| \frac{p}{V_{\infty}^1(\tau_{\infty}^1(p))}-f(\Lambda_1) \right| \geq \eps \right) \leq \eps.\]
We fix such a $p$ until the end of the proof. Moreover, we have
\[ \frac{q}{V^1_{\infty}(\tau^1_{\infty}(q))-V^1_{\infty}(\tau^1_{\infty}(p))} \xrightarrow[q \to +\infty]{a.s.} f(\Lambda_1),\]
so for $q$ large enough, we have
\[ \P \left( \left| \frac{q}{V^1_{\infty}(\tau^1_{\infty}(q))-V^1_{\infty}(\tau^1_{\infty}(p))} - f(\Lambda_1) \right| \geq \eps \right) \leq \eps.\]
We fix a such a $q>p$ until the end of the proof. 
Moreover, the local convergence of $(T_{n,g_n},e_n^1)$ to $\TT^1_{\Lambda_1}$ implies the convergence of the peeling exploration. Hence, the probability that $\tau_n^1(p)$ is finite goes to $1$ as $n \to +\infty$. Therefore, for $n$ large enough, we have
\begin{equation}\label{eqn_tau_p_vs_infty}
\P \left( \left| \frac{p}{V_n^1(\tau_n^1(p))}-f(\Lambda_1) \right| \geq 2\eps \right) \leq 2\eps,
\end{equation}
\begin{equation}\label{eqn_tau_q_vs_infty}
\P \left( \left| \frac{q}{V_n^1(\tau_n^1(q))-V_n^1(\tau_n^1(p))}-f(\Lambda_1) \right| \geq 2\eps \right) \leq 2\eps.
\end{equation}
By combining \eqref{eqn_tau_p_vs_infty} and \eqref{eqn_tau_q_vs_infty}, we also obtain
\begin{equation}\label{eqn_tau_q_vs_tau_p}
\P \left( \left| \frac{q}{V_n^1(\tau_n^1(q))-V_n^1(\tau_n^1(p))}-\frac{p}{V_n^1(\tau_n^1(p))} \right| \geq 4\eps \right) \leq 4\epsilon
\end{equation}
for $n$ large enough. Moreover, the same is true if we replace the exploration from $e_n^1$ by the exploration from $e_n^2$ (with of course $\Lambda_2$ playing the role of $\Lambda_1$).

We now specify the properties that we need our peeling algorithm $\mathcal{A}$ to satisfy. It roughly means that the edges that we peel after time $\tau(p)$ should not depend on what happened before time $\tau(p)$. More precisely, before the time $\tau(p)$, the peeled edge can be any edge of the boundary. At time $\tau(p)$, we color in red an edge of the boundary of $\expl(\tau(p))$ according to some deterministic convention. At time $i \geq \tau(p)$, the choice of the edge to peel at time $i+1$ must only depend on the triangulation $\expl(i) \backslash \expl(\tau(p))$, which is rooted at the red edge (see Figure \ref{fig_peeling_algo}). It is easy to see that such an algorithm exists: we only need to fix a peeling algorithm $\mathcal{A}'$ for triangulations of the sphere or the plane, and a peeling algorithm $\mathcal{A}''$ for triangulations of the $p$-gon, and to use $\mathcal{A}'$ before time $\tau(p)$ and $\mathcal{A}''$ after $\tau(p)$. Note also that we can know, by only looking at $\expl(i)$, if $i \leq \tau(p)$ or not, so at each step the peeled edge will indeed depend only on the explored part.

\begin{figure}
\begin{center}
\includegraphics[scale=0.6]{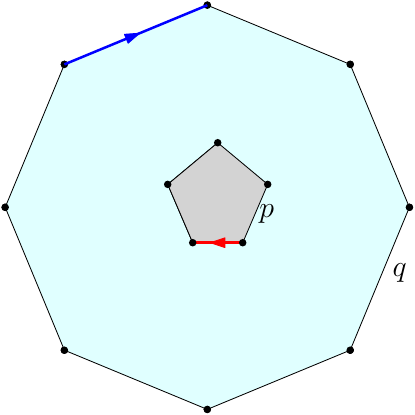}
\end{center}
\caption{Illustration of the choice of the peeling algorithm: in gray, the part $\expl(\tau(p))$. In cyan, the part $\expl(i) \backslash \expl(\tau(p))$, rooted at the red edge. The peeled edge (in blue) may only depend on the cyan part (which is rooted at the red edge).}\label{fig_peeling_algo}
\end{figure}

We can now define precisely our surgery operation on $\left( T_{n,g_n}, e_n^1, e_n^2 \right)$. We say that $e_n^1$ and $e_n^2$ are \emph{well separated} if $\tau_n^1(p), \tau_n^2(p)<+\infty$ and if the regions $\expl_n^1(\tau_n^1(p))$ and $\expl_n^2(\tau_n^2(p))$ have no common vertex. If $e_n^1$ and $e_n^2$ are well separated, we remove the triangulations $\expl_n^1(\tau_n^1(p))$ and $\expl_n^2(\tau_n^2(p))$, which creates two holes of perimeter $p$. We then glue each of the two triangulations to the hole where the other was, in such a way that the red edges of our peeling algorithm coincide (see Figure \ref{fig_two_holes}). If the two roots are not well-separated, then $\left( T_{n,g_n}, e_n^1, e_n^2 \right)$ is left unchanged. This operation is an involution on the set of bi-rooted triangulations with fixed size and genus. Therefore, if we denote by $T'_{n,g_n}$ the new triangulation we obtain, it is still uniform, and $\left( T'_{n,g_n}, e_n^1, e_n^2 \right)$ has the same distribution as $\left( T_{n,g_n}, e_n^1, e_n^2 \right)$.

\begin{lem}
The probability that $e_n^1$ and $e_n^2$ are well separated goes to $1$ as $n \to +\infty$.
\end{lem}

\begin{proof}
The local convergence of the exploration implies that the sequence $\left( \tau^1_n(p) \right)_{n \geq 0}$ is tight, and so is $\left( V_n^1 (\tau^1_n(p)) \right)_{n \geq 0}$. Since the diameter of a triangulation is bounded by its volume, the diameters of the maps $\expl_n^1 (\tau^1_n(p))$ are tight (when $n \to +\infty$), and the same is true around $e_n^2$. On the other hand, by local convergence (for $d_{\loc}$), for every fixed $r$, the volume of the ball of radius $r$ around $e_n^1$ is tight. Since $e_n^2$ is uniform and independent of $e_n^1$, we have
\[\P \left( d_{T_{n,g_n}} (e^1_n,e^2_n) \leq r \right) \xrightarrow[n \to +\infty]{} 0\]
for every $r \geq 0$. From here, we obtain
\[ \P \left( d_{T_{n,g_n}} (e^1_n,e^2_n) \leq \diam \left( \expl_n^1 (\tau^1_n(p)) \right) + \diam \left( \expl_n^2 (\tau^2_n(p)) \right) \right) \xrightarrow[n \to +\infty]{} 0. \]
If this last event occurs, then the regions $\expl_n^1 (\tau^1_n(p))$ and $\expl_n^2 (\tau^2_n(p))$ do not intersect, which proves the lemma.
\end{proof}

Now, if we perform a peeling exploration on $T'_{n,g_n}$ from $e_n^1$ using algorithm $\mathcal{A}$, then the part explored at time $\tau(p)$ will be exactly the triangulation $\expl^1_n(\tau^1_n(p))$. Moreover, the red edge on $\partial \expl^1_n(\tau^1_n(p))$ is glued at the same place as the red edge on $\partial \expl^2_n(\tau^2_n(p))$ and our peeling algorithm "forgets" the interior of $\expl^1_n(\tau^1_n(p))$. Hence, the part explored between $\tau(p)$ and $\tau(q)$ is exactly $\expl^2_n(\tau^2_n(q)) \backslash \expl^2_n(\tau^2_n(p))$ (this is where we need the algorithm $\mathcal{A}$ to satisfy the property described above). Therefore, the part of $T'_{n,g_n}$ discovered between times $\tau(p)$ and $\tau(q)$ has perimeter $q$ and volume
\[ V_n^2(\tau_n^2(q))-V_n^2(\tau_n^2(p)). \]
Now, since $(T'_{n,g_n}, e_n^1)$ has the same distribution as $(T_{n,g_n}, e_n^1)$, we can apply \eqref{eqn_tau_q_vs_tau_p} to the exploration of $T'_{n,g_n}$ from $e^1_n$. We obtain
\[\P \left( \left| \frac{q}{V_n^2(\tau_n^2(q))-V_n^2(\tau_n^2(p))}-\frac{p}{V_n^1(\tau_n^1(p))} \right| \geq 4\eps \right) \leq 4\epsilon.\]
By combining this with \eqref{eqn_tau_p_vs_infty} for the exploration of $T_{n,g_n}$ started from $e_n^1$ and with \eqref{eqn_tau_q_vs_infty} for the exploration of $T_{n,g_n}$ started from $e_n^2$, we obtain
\[\P \left( \left| f(\Lambda_1)-f(\Lambda_2) \right| \geq 8\eps \right) \leq 8\eps.\]
This is true for any $\eps>0$, so we have $f(\Lambda_1)=f(\Lambda_2)$ a.s. Since $f$ is strictly decreasing on $(0,\Lambda_c]$, we are done.
\end{proof}

\subsection{Conclusion}

In order to conclude the proof of Theorem \ref{main_thm}, we will need to compute the mean inverse root degree in the PSHT. This is the only observable that is easy to compute in $T_{n,g_n}$, so it is not surprising that such a result is needed to link $\theta$ to $\lambda$.

\begin{prop}\label{prop_degree_pshit}
For $\lambda \in (0,\lambda_c]$, let
\[ d(\lambda)=\E \left[ \frac{1}{\deg_{\TT_{\lambda}}(\rho)} \right].\]
Then we have
\[d(\lambda)=\frac{  h \log \frac { 1 + \sqrt { 1 - 4 h } } { 1 - \sqrt { 1 - 4 h } } } { ( 1 + 8 h ) \sqrt { 1 - 4 h } }, \]
where $h$ is given by \eqref{eqn_defn_h}. In particular, the function $d$ is strictly increasing on $(0,\lambda_c]$, with $d(\lambda_c)=\frac{1}{6}$ and $\lim_{\lambda \to 0} d(\lambda)=0$.
\end{prop}

We postpone the proof of Proposition \ref{prop_degree_pshit} until Section \ref{subsec_root_degree}, and finish the proof of Theorem \ref{main_thm}. We recall that we work with a subsequential limit of $(T_{n,g_n})$, and we know that it has the same distribution as $\TT_{\Lambda}$ for some random $\Lambda$. Our goal is to prove that $\Lambda$ is the solution to $d(\Lambda)=\frac{1-2\theta}{6}$. Note that by Proposition \ref{prop_degree_pshit}, the solution to this equation exists and is unique.

The idea is the following: by Proposition \ref{prop_same_map_same_lambda}, the random parameter $\Lambda$ only depends on the triangulation $T_{n,g_n}$ and not on its root. On the one hand, for any triangulation, the mean inverse root degree over all possible choices of the root is $\frac{1-2\theta}{6}$. Therefore, the mean inverse root degree over all triangulations corresponding to a fixed value of $\Lambda$ is $\frac{1-2\theta}{6}$. On the other hand, the mean inverse degree conditionally on $\Lambda$ is $d(\Lambda)$, so we should have $d(\Lambda)=\frac{1-2\theta}{6}$.

To prove this properly, we need a way to "read" $\Lambda$ on finite maps, which is the goal of the next (easy) technical lemma. As earlier, we restrict ourselves to a subset of values of $n$ along which $T_{n,g_n} \to \TT_{\Lambda}$.

\begin{lem}\label{lem_lambda_finite}
There is a function $\ell : \bigcup_{n \geq 1} \T(n,g_n) \longrightarrow [0,\lambda_c]$ such that we have the convergence
\[ \left( T_{n,g_n}, \ell(T_{n,g_n}) \right) \xrightarrow[n \to +\infty]{(d)} \left( \TT_{\Lambda}, \Lambda \right).  \]
\end{lem}

Note that a priori $\ell (t)$ may depend on the choice of the root of $t$, although Proposition \ref{prop_same_map_same_lambda} will guarantee that this is asymptotically not the case.

\begin{proof}
By the Skorokhod representation theorem, we may assume the convergence $T_{n,g_n} \to \TT_{\Lambda}$ is almost sure. As explained in the beginning of Section \ref{sec_ergodicity}, the parameter $\Lambda$ is a measurable function of $\TT_{\Lambda}$, so there is a measurable function $\widetilde{\ell}$ from the set of all (finite or infinite) triangulations such that $\widetilde{\ell}(\TT_{\Lambda})=\Lambda$ a.s. Therefore, for any $\eps>0$, we can find a continuous function (for the local topology) $\ell_{\eps}$ such that
\[ \P \left( \left| \ell_{\eps}(\TT_{\Lambda})-\Lambda \right| \geq \frac{\eps}{2} \right) \leq \frac{\eps}{2}.\]
For $n$ larger than some $N_{\eps}$, we have
\[ \P \left( \left| \ell_{\eps}(T_{n,g_n})-\Lambda \right| \geq \eps \right) \leq \eps.\]
Now let $\eps_k \to 0$ as $k \to +\infty$. We may assume that $(N_{\eps_k})$ is strictly increasing. For any triangulation $t \in \T(n,g_n)$, we set $\ell(t)=\ell_{\eps_k}(t)$, where $k$ is such that $N_{\eps_k} < n \leq N_{\eps_{k+1}}$. This ensures $\ell(T_{n,g_n}) \to \Lambda$ almost surely, so
\[ \left( T_{n,g_n}, \ell(T_{n,g_n}) \right) \xrightarrow[n \to +\infty]{(d)} \left( \TT_{\Lambda}, \Lambda \right).  \]
\end{proof}

We are now ready to prove our main Theorem.

\begin{proof}[Proof of Theorem \ref{main_thm}]
Let $f$ be a continuous, bounded function on $[0,\lambda_c]$. For any $n$, let $e_n^1$ and $e_n^2$ be two oriented edges chosen independently and uniformly in $T_{n,g_n}$, and let $\rho_n^1$ and $\rho_n^2$ be their starting points. We will estimate in two different ways the limit of the quantity
\begin{equation}\label{eqn_inverse_degree_and_f}
\E \left[ \frac{1}{\deg(\rho_n^1)} f(\ell(T_{n,g_n},e_n^1)) \right].
\end{equation}
On the one hand, the quantity in the expectation is bounded and converges in distribution to $\frac{1}{\deg(\rho)} f(\Lambda)$, where $\rho$ is the root vertex of $\TT_{\Lambda}$, so \eqref{eqn_inverse_degree_and_f} goes to
\[ \E \left[ \frac{1}{\deg(\rho)} f(\Lambda) \right] = \E \left[ d(\Lambda) f(\Lambda) \right]\]
as $n \to +\infty$. On the other hand, by Proposition \ref{prop_same_map_same_lambda} and Lemma \ref{lem_lambda_finite}, we have
\[ \left( \ell(T_{n,g_n}, e_n^1), \ell(T_{n,g_n}, e_n^2) \right) \xrightarrow[n \to +\infty]{(d)} \left( \Lambda, \Lambda \right), \]
so $\ell(T_{n,g_n}, e_n^1) - \ell(T_{n,g_n}, e_n^2)$ goes to $0$ in probability. Since $f$ is bounded and uniformly continuous, the difference $f \left( \ell(T_{n,g_n}, e_n^1) \right) - f \left( \ell(T_{n,g_n}, e_n^2) \right)$ goes to $0$ in $L^1$, so we have
\[ \E \left[ \frac{1}{\deg(\rho_n^1)} f(\ell(T_{n,g_n},e_n^1)) \right] = \E \left[ \frac{1}{\deg(\rho_n^1)} f(\ell(T_{n,g_n},e_n^2)) \right] +o(1).\]
Moreover, the expectation of $\frac{1}{\deg(\rho_n^1)}$ conditionally on $(T_{n,g_n},e_n^2)$ is equal to $\frac{|V(T_{n,g_n})|}{2|E(T_{n,g_n})|}=\frac{n+2-2g_n}{6n}$, so we can write
\[ \E \left[ \frac{1}{\deg(\rho_n^1)} f(\ell(T_{n,g_n},e_n^1)) \right] = \frac{n+2-2g_n}{6n} \E \left[ f(\ell(T_{n,g_n},e_n^2)) \right] +o(1) \xrightarrow[n \to +\infty]{} \frac{1-2\theta}{6} \E \left[ f(\Lambda) \right].\]
Therefore, for any bounded, continuous $f$ on $[0,\lambda_c]$, we have
\[ \E \left[ d(\Lambda) f(\Lambda) \right]=\frac{1-2\theta}{6} \E \left[ f(\Lambda) \right],\]
so $d(\Lambda)=\frac{1-2\theta}{6}$ a.s. By injectivity of $d$, this fixes a deterministic value for $\Lambda$ that depends only on $\theta$, which concludes the proof.
\end{proof}

\subsection{The average root degree in the type-I PSHT via uniform spanning forests}
\label{subsec_root_degree}

Our goal is now to prove Proposition \ref{prop_degree_pshit}. Since we have not been able to perform a direct computation, our strategy will be to use the similar computation for the type-II PSHIT that is done in \cite[Appendix B]{B18these}. To link the mean degree in the type-I and in the type-II PSHT, we use the core decomposition of the type-I PSHT \cite[Appendix A]{B18these} and an interpretation in terms of the Wired Uniform Spanning Forest.

\paragraph{The type-II case.}
We will need to use the computation of the same quantity in the type-II case, which is performed in \cite[Appendix B]{B18these}. We first recall briefly the definition of the type-II PSHT \cite{CurPSHIT}. These are the type-II analogues of the type-I PSHT, i.e. they may contain multiple edges but no loop. They form a one-parameter family $\left( \TT_{\kappa}^{II} \right)_{0 < \kappa \leq \kappa_c}$ of random infinite triangulations of the plane, where $\kappa_c=\frac{2}{27}$. Their distribution is characterized as follows: for any type-II triangulation with a hole of perimeter $p$ and $v$ vertices in total, we have
\[ \P \left( t \subset \TT^{II}_{\kappa} \right) = C_p^{II}(\kappa) \times \kappa^v,\]
where the numbers $C_p^{II}(\kappa)$ are explicitely determined by $\kappa$.

\begin{prop}\cite[Appendix B]{B18these}\label{prop_degree_II}
For $\kappa \in (0,\kappa_c]$, let
\[ d_{II}(\kappa)=\E \left[ \frac{1}{\deg_{\TT^{II}_{\kappa}}(\rho)} \right].\]
Then we have
\[d_{II}(\kappa)=-\frac{1-\alpha}{2} + \frac{(1-\alpha)\sqrt{\alpha}}{2 \sqrt{3\alpha-2}} \log \frac{\sqrt{\alpha}+\sqrt{3\alpha-2}}{\sqrt{\alpha}-\sqrt{3\alpha-2}},\]
where $\alpha \in \left[ \frac{2}{3}, 1 \right)$ is such that $\kappa=\frac{\alpha^2(1-\alpha)}{2}$.
\end{prop}

In \cite{B18these}, Proposition \ref{prop_degree_II} is proved by using a peeling exploration to obtain an equation on the generating function of the mean inverse root degree in PSHT with boundaries. While this approach could theoretically also work in the type-I setting, some technical details make the formulas much more complicated. Although everything remains in theory completely solvable, our efforts to push the computation until the end have failed, even with a computer algebra software.
Therefore, we will rely on Proposition \ref{prop_degree_II} and on the correspondence between type-I PSHT and type-II PSHT given in \cite[Appendix A]{B18these}.

\paragraph{The type-I--type-II correspondence.}
We now introduce the correspondence between the type-I and the type-II PSHT stated in \cite[Appendix A]{B18these}. We write $\beta=1-\frac{1+2h}{\sqrt{1+8h}}$, and $\kappa=\frac{h}{(1+2h)^3}$, where $h$ is given by \eqref{eqn_defn_h}. We also define a random triangulation $A_{\lambda}$ of the $2$-gon as follows. We start from two vertices $x$ and $y$ linked by $B$ edges, where $B-1$ is a geometric variable of parameter $\beta$. In each of the $B-1$ slots between these edges, we insert a loop (the loops are glued to $x$ or $y$ according to independent fair coin flips). Finally, these $B-1$ loops are filled with i.i.d. Boltzmann triangulations of the $1$-gon with parameter $\lambda$ (see Figure \ref{fig_a_lambda}). Copies of $A_{\lambda}$ will be called \emph{blocks} in what follows.

\begin{figure}
\centering
\includegraphics[scale=0.6]{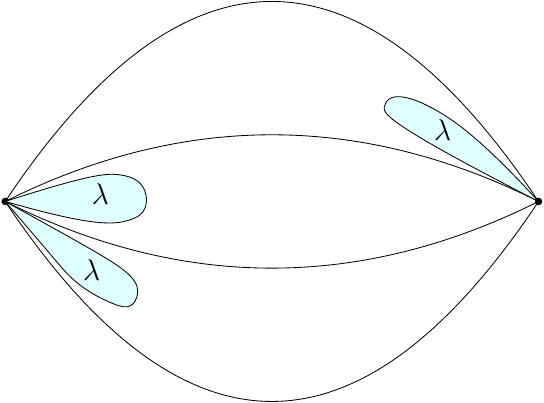}
\caption{The random block $A_{\lambda}$ (here we have $B=4$). The green parts are i.i.d. $\lambda$-Boltzmann triangulations of the $1$-gon.}\label{fig_a_lambda}
\end{figure}

Let $|A_{\lambda}|$ be the number of inner vertices of $A_{\lambda}$ (i.e. $x$ and $y$ excluded), and let $E(A_{\lambda})$ be its set of edges (including the boundary edges). The Euler formula implies $|E(A_{\lambda})|=1+3|A_{\lambda}|$. Let $\widetilde{A}_{\lambda}$ be a random triangulation with the same distribution as $A_{\lambda}$, biased by $|E(A_{\lambda})|$. If $T$ is an infinite type-II triangulation of the plane and $(a^e)_{e \in E(T)}$ are blocks, we denote by $\Phi \left( T, \left( a^e \right) \right)$ the triangulation obtained from $T$ by replacing each edge $e$ of $T$ by $a^e$. After elementary computations, the correspondence shown in \cite[Appendix A]{B18these} can be reformulated as follows.

\begin{prop}\label{prop_type_I-II}
The triangulation $\TT_{\lambda}$ has the same law as $\Phi \left( \TT^{II}_{\kappa}, \left( A^e_{\lambda} \right)_{e \in E(\TT^{II}_{\kappa})} \right)$, where conditionally on $\TT^{II}_{\kappa}$:
\begin{itemize}
\item
the blocks $A^e_{\lambda}$ are independent;
\item
the block $A^{e_0^{II}}_{\lambda}$, where $e_0^{II}$ is the root edge of $\TT^{II}_{\kappa}$, has the law of $\widetilde{A}_{\lambda}$, rooted at a uniform oriented edge of $\widetilde{A}_{\lambda}$;
\item
for $e \ne e_0^{II}$, the block $A^e_{\lambda}$ has the law of $A_{\lambda}$.
\end{itemize}
\end{prop}

In the remainder of this section, we will consider that $\TT_{\lambda}$ and $\TT^{II}_{\kappa}$ are coupled together as described by Proposition \ref{prop_type_I-II}. It will also be useful to consider the triangulation obtained from $\TT^{II}_{\kappa}$ in a similar way, but where $A^{e_0^{II}}_{\lambda}$ has the law of $A_{\lambda}$ instead of $\widetilde{A}_{\lambda}$. We denote this triangulation by $\TT'_{\lambda}$. We note right now that $\TT_{\lambda}$ is absolutely continuous with respect to $\TT'_{\lambda}$.

\paragraph{The wired uniform spanning forest.}
The last ingredient that we need to prove Proposition \ref{prop_degree_pshit} is the oriented wired uniform spanning forest (OWUSF) \cite{BLPS01}, which may be defined on any transient graph. We first recall its definition. We recall that if $(\gamma_i)$ is a transient path, its \emph{loop-erasure} is the path $(\gamma_{i_j})$, where $i_0=0$ and $i_{j+1}=1+\max \{ i \geq i_j | \gamma(i)=\gamma(i_j) \}$. If $G$ is an infinite transient graph, its OWUSF $\overrightarrow{F}$ is generated by the following generalization of the Wilson algorithm \cite{W96}.
\begin{enumerate}
\item
We order the vertices of $G$ as $(v_n)_{n \geq 1}$.
\item
We run a simple random walk $(X_n)$ on $G$ started from $v_1$, and consider its loop-erasure, which is an oriented path from $v_1$ to $\infty$. We add this path to $\overrightarrow{F}$.
\item
We consider the first $i$ such that $v_i$ does not belong to any edge of $\overrightarrow{F}$ yet. We run a SRW from $v_i$, stopped when it hits $\overrightarrow{F}$, and add its loop-erasure to $\overrightarrow{F}$.
\item
We repeat step 3 infinitely many times.
\end{enumerate}
It is immediate that $\overrightarrow{F}$ is an oriented forest such that for any vertex $v$ of $G$, there is exactly one oriented edge of $\overrightarrow{F}$ going out of $v$. It can be checked that the distribution of $\overrightarrow{F}$ does not depend on the way the vertices are ordered.

\paragraph{Stationarity and reversibility.}
We recall that a random rooted graph $G$ is \emph{reversible} if its distribution is invariant under reversing the orientation of the root edge. It is \emph{stationary} if its distribution is invariant under re-rooting $G$ at the first step of the simple random walk. In particular, we recall from \cite{CurPSHIT} that the type-II PSHT are stationary and reversible. Moreover, the proof of \cite{CurPSHIT} also works for the type-I PSHT.
Similarly, a random rooted, forest-decorated graph $\left( G, \overrightarrow{e}_0, \overrightarrow{F} \right)$ is called \emph{reversible} (resp. stationary) if its distribution is invariant under reversing the root edge (resp. rerooting at the first step of the SRW).

If two random, rooted transient graphs $(G, \overrightarrow{e}_0)$ and $(G', \overrightarrow{e}'_0)$ have the same distribution and if $\overrightarrow{F}$ and $\overrightarrow{F}'$ are their respective OWUSF, then $(G, \overrightarrow{e}_0, \overrightarrow{F})$ and $(G', \overrightarrow{e}'_0, \overrightarrow{F}')$ also have the same distribution. It follows that if $(G, \overrightarrow{e}_0)$ is stationary and reversible, so is the triplet $(G, \overrightarrow{e}_0, \overrightarrow{F})$.
The link between the OWUSF and the mean inverse root degree is given by the next lemma.

\begin{lem}\label{lem_owusf_and_degree}
Let $(G,\overrightarrow{e_0})$ be a stationnary, reversible and transient random rooted graph, and let $\rho$ be the starting point of the root edge $\overrightarrow{e_0}$. We denote by $\overrightarrow{F}$ the OWUSF of $G$. Then
\[\P \left( \overrightarrow{e_0} \in \overrightarrow{F} \right)= \, \E \left[ \frac{1}{\deg_G(\rho)} \right].\]
\end{lem}

\begin{proof}
Conditionally on $G$, let $\overrightarrow{e_1}$ be a random oriented edge chosen uniformly (independently of $\overrightarrow{F}$) among all the edges leaving the root vertex. By invariance and reversibility of $(G, \overrightarrow{e_0})$, we know that $(G, \overrightarrow{e_1}, \overrightarrow{F})$ has the same distribution as $(G, \overrightarrow{e_0}, \overrightarrow{F})$. Therefore, we have
\[\P \left( \overrightarrow{e_0} \in \overrightarrow{F} \right) = \P \left( \overrightarrow{e_1} \in \overrightarrow{F} \right) = \E \left[ \P \left( \overrightarrow{e_1} \in \overrightarrow{F} | G,\overrightarrow{e_0}, \overrightarrow{F} \right) \right] = \E \left[ \frac{1}{\deg_G(\rho)} \right],\]
where in the end we use the fact that, by construction, $\overrightarrow{F}$ contains exactly one oriented edge going out of the root vertex.
\end{proof}

\begin{rem}
Lemma \ref{lem_owusf_and_degree} is the analogue for stationary, reversible graphs of the fact that the WUSF on any unimodular random graph has expected root degree $2$ \cite[Theorem 7.3]{AL07}. These two properties can also be deduced from one another.
\end{rem}

\begin{proof}[Proof of Proposition \ref{prop_degree_pshit}]
By an easy uniform integrability argument, the function $d$ is continuous in $\lambda$, so it is enough to prove the formula for $\lambda<\lambda_c$. In this case, all variants of $\TT_{\lambda}$ that we will consider are a.s. transient by results from \cite{CurPSHIT}, so the OWUSF on these graphs is well defined. We now fix $\lambda \in (0,\lambda_c)$. Let $\beta$ and $\kappa$ be given by Proposition \ref{prop_type_I-II}. We will denote by $\overrightarrow{e}_0$, $\overrightarrow{e}'_0$ and $\overrightarrow{e}_0^{II}$ the respective root edges of $\TT_{\lambda}$,  $\TT'_{\lambda}$ and $\Tii$, and by $\overrightarrow{F}$, $\overrightarrow{F}'$ and $\overrightarrow{F}^{II}$ their respective OWUSF.

By Proposition \ref{prop_type_I-II}, conditionally on $\Tii$ and on the blocks $(A_{\lambda}^e)_{e \in E(\Tii)}$, the edge $\overrightarrow{e}_0$ is picked uniformly among the oriented edges of $A_{\lambda}^{e_0^{II}}$. Moreover, this choice is also independent of $\overrightarrow{F}$. Therefore, we have
\[ \P \left( \overrightarrow{e_0} \in \overrightarrow{F} \right) = \E \left[ \P \left( \overrightarrow{e_0} \in \overrightarrow{F} | \Tii, (A^e_{\lambda}), \overrightarrow{F} \right) \right] = \E \left[ \frac{|E(\widetilde{A}_{\lambda}) \cap \overrightarrow{F}|}{2 |E(\widetilde{A}_{\lambda})|} \right], \]
where we recall that $E(\widetilde{A}_{\lambda})$ is the set of edges of $\widetilde{A}_{\lambda}$. Moreover, the pair $(\TT_{\lambda}, \overrightarrow{F})$ has the same distribution as $(\TT'_{\lambda}, \overrightarrow{F}')$ biased by $|E(A_{\lambda}^{e_0})|$. Therefore, we can write
\begin{equation}\label{formula_in_T'}
\P \left( \overrightarrow{e_0} \in \overrightarrow{F} \right) = \frac{\E \left[ |E(A_{\lambda}^{e_0}) \cap \overrightarrow{F}'| \right] }{2 \E \left[ |E(A_{\lambda})| \right]}.
\end{equation}
Since the denominator can easily be explicitely computed by using the definition of the block $A_{\lambda}$, we first focus on the numerator. In a block, we call \emph{principal edges} the edges joining the two boundary vertices of the block. We know that there is exactly one edge of $\overrightarrow{F}'$ going out of every inner vertex of $A_{\lambda}$. Moreover, if a non-principal edge of $A_{\lambda}$ belongs to $\overrightarrow{F}'$, then it goes out of an inner vertex (it cannot go out of a boundary vertex since the edges of $\overrightarrow{F}'$ are "directed towards infinity"). Therefore, the number of non-principal edges of $A_{\lambda}$ that belong to $\overrightarrow{F}'$ is exactly $|A_{\lambda}|$.
On the other hand, since $\overrightarrow{F}'$ is a forest, it contains at most one principal edge of $A_{\lambda}$. Therefore, we have
\[ \E \left[ |E(A_{\lambda}^{e_0}) \cap \overrightarrow{F}'| \right]=\E[|A_{\lambda}|] + \P \left( \overrightarrow{F}' \mbox{ contains a principal edge of $A_{\lambda}^{e_0}$} \right). \]

To finish the computation, we claim that
\begin{equation}\label{eqn_principal_edge}
\P \left( \overrightarrow{F}' \mbox{ contains a principal edge of $A_{\lambda}^{e_0}$} \right) = 2\E \left[ \frac{1}{\deg_{\Tii} (\rho)} \right]=2d_{II}(\kappa).
\end{equation}
Indeed, let $\rho^{II}$ be the starting vertex of $\overrightarrow{e}_0^{II}$. By reversibility of $\Tii$, the left-hand side of \eqref{eqn_principal_edge} is twice the probability that $\overrightarrow{F}'$ contains a principal edge of $A_{\lambda}^{e_0}$ going out of $\rho^{II}$. Moreover, since $\Tii$ is stationary, the distribution of $\TT'_{\lambda}$ is stable under the following operation:
\begin{itemize}
\item[$\bullet$]
first, we resample the root edge of $\Tii$ uniformly among all the oriented edges going out of $\rho^{II}$;
\item[$\bullet$]
then, we pick the new root edge of $\TT'_{\lambda}$ uniformly among all the oriented edges of the block corresponding to the new root edge of $\Tii$.
\end{itemize}
Note that this operation is different from resampling the root of $\TT'_{\lambda}$ uniformly among the edges going out of its root vertex, and that $\TT'_{\lambda}$ is \emph{not} stationary.

Since $\TT'_{\lambda}$ is stationary for this operation, so is $\left( \TT'_{\lambda}, \overrightarrow{F}' \right)$. On the other hand, there is exactly one block incident to $\rho^{II}$ in which a principal edge going out of $\rho^{II}$ belongs to $\overrightarrow{F}'$. By the same argument as in the proof of Lemma \ref{lem_owusf_and_degree}, this implies \eqref{eqn_principal_edge}. By combining Lemma \ref{lem_owusf_and_degree} with \eqref{formula_in_T'} and \eqref{eqn_principal_edge}, we obtain
\[ d(\lambda)=\frac{\E[|A_{\lambda}|]+2d_{II}(\kappa)}{2 \left( 1+3\E[|A_{\lambda}|] \right)}. \]
By the definition of $A_{\lambda}$, we also have
\[ \E[|A_{\lambda}|] = \frac{\beta}{1-\beta} \times \frac{\lambda Z'_1(\lambda)}{Z_1(\lambda)} = \frac{2h}{1+2h} ,\]
where $Z_1(\lambda)$ is the partition function of $\lambda$-Boltzmann type-I triangulations of a $1$-gon. Finally, the value of $d_{II}(\kappa)$ is given by Proposition \ref{prop_degree_II}. It is easy to obtain that the $\alpha$ of Proposition \ref{prop_degree_II} is equal to $\frac{1}{1+2h}$, and we get the formula for $d(\lambda)$.

It remains to prove that this is an increasing function of $\lambda$ (or equivalently of $h$). For this, set $x=\sqrt{1-4h}$. We want to prove that the function
\[ x \to \frac{1-x^2}{4x(3-2x^2)} \log \frac{1+x}{1-x}\]
is decreasing on $[0,1]$. By computing the derivative with respect to $x$, this is equivalent to
\[ (2x^4-3x^2+3)x\log \frac{1+x}{1-x} \geq 2x^2(3-2
x^2)\]
for $0<x<1$. The power series expansion of the left-hand side is
\[6x^2-4x^4+\sum_{k \geq 3} \left( \frac{6}{2k-1}-\frac{6}{2k-3}+\frac{4}{2k-5} \right)x^{2k} \geq 6x^2-4x^4.\]
Since all the terms in the sum are positive, this proves monotonicity.
\end{proof}

\section{Combinatorial asymptotics}

The goal of this section is to prove Theorem \ref{thm_asympto}. The proof relies on the lemma below, which is an easy consequence of Theorem \ref{main_thm}. We recall that $\tau(n,g)$ is the number of triangulations of genus $g$ with $2n$ faces, and that $\lambda(\theta)$ is the solution to \eqref{eqn_lambda_vs_theta}.

\begin{lem}\label{prop_rapport_voisins}
Let $(g_n)$ be such that $\frac{g_n}{n}\rightarrow \theta$, with $0\leq \theta <\frac{1}{2}$.
Then 
\[\frac{\tau(n-1,g_n)}{\tau(n,g_n)}\rightarrow \lambda(\theta)\]
\end{lem}

\begin{proof}
This is a simple consequence of Theorem \ref{main_thm}. Indeed, we have $(T_{n,g_n})\rightarrow \TT_{\lambda(\theta)}$ locally. Let $t_1$ be the triangulation represented on Figure \ref{fig_one_gon}. We have on the one hand
\[ \P \left( t_1 \subset T_{n,g_n} \right)= \frac{\tau_1(n-1,g_n)}{\tau(n,g_n)}=\frac{\tau(n-1,g_n)}{\tau(n,g_n)}\]
by the usual root transformation (see Remark \ref{rem_root_transfo}). On the other hand, we have
\[\P \left( t_1 \subset \TT_{\lambda(\theta)} \right) = C_1(\lambda(\theta))\lambda(\theta)^2=\lambda(\theta).\]
Since $\{ t_1 \subset T\}$ is closed and open for $d_{\loc}$, the lemma follows.

\begin{figure}%[!h]
\centering
\includegraphics[scale=0.5]{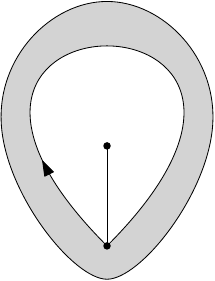}
\caption{The triangulation $t_1$ with perimeter $1$ and volume $2$.}
\label{fig_one_gon}
\end{figure}
\end{proof}

The idea of the proof of Theorem \ref{thm_asympto} is then to write
\begin{equation}\label{eqn_decomposition_taung}
\tau(n,g_n)=\frac{\tau(n,g_n)}{\tau((2+\eps)g_n,g_n)} \times \tau((2+\eps)g_n,g_n)
\end{equation}
for some small $\eps>0$. The first factor can be turned into a telescopic product and estimated by Lemma \ref{prop_rapport_voisins}. The trickier part will be to estimate the second. We will prove the following bounds.

\begin{prop}\label{prop_near_diagonal}
There is a function $h$ with $h(\eps)\to 0$ as $\eps \to 0$ such that
\[e^{o(g)} \left(\frac{6}{e} \right)^{2g} (2g)^{2g} \leq \tau((2+\eps)g,g) \leq e^{h(\eps) g+o(g)} \left(\frac{6}{e} \right)^{2g} (2g)^{2g}\]
as $g \to +\infty$. Moreover, if $\eps_g \to 0$, then
\[ \tau((2+\eps_g)g,g)=e^{o(g)} \left(\frac{6}{e} \right)^{2g} (2g)^{2g}.\]
\end{prop}

We delay the proof of Proposition \ref{prop_near_diagonal}, and first explain how to finish the proof of Theorem \ref{thm_asympto}.

\begin{proof}[Proof of Theorem \ref{thm_asympto}]
We first note that $f(1/2)=\frac{6}{e}$, so the case $\theta=1/2$ is exactly the second part of Proposition~\ref{prop_near_diagonal}. We now assume $0 \leq \theta <1/2$. Let $\eps>0$ be such that $\theta<\frac{1}{2+\eps}$. We estimate the first factor of \eqref{eqn_decomposition_taung}. We write
\[ \frac{1}{n} \log\left(\frac{\tau(n,g_n)}{\tau((2+\eps)g_n,g_n)}\right) = \frac{1}{n} \sum_{i=(2+\eps)g_n}^n \log\left(\frac{\tau(i,g_n)}{\tau(i-1,g_n)}\right).\]
By Lemma~\ref{prop_rapport_voisins}, when $\frac{g_n}{i} \to \frac{1}{t}$, so $\log\frac{\tau(i,g_n)}{\tau(i-1,g_n)} \to \log \frac{1}{\lambda(1/t)}$. Moreover, by the bounded ratio lemma (Lemma~\ref{ratio_lem}), each of the terms is bounded by some constant $C_{\eps}$. Hence, by dominated convergence, we have
\begin{equation}\label{eqn_telescopic_product}
\frac{1}{g_n} \log\left(\frac{\tau(n,g_n)}{\tau((2+\eps)g_n,g_n)}\right) \xrightarrow[n \to +\infty]{} \int_{(2+\eps)}^{1/\theta} \log \frac{1}{\lambda(1/t)} \mathrm{d}t.
\end{equation}
On the other hand, if we replace $g$ by $g_n$ with $\frac{g_n}{n} \to \theta$, then the left-hand side of Proposition~\ref{prop_near_diagonal} becomes
$n^{2g_n} \exp \left( \left( 2\theta \log \frac{12\theta}{e} \right) n +o(n) \right)$,
so Proposition~\ref{prop_near_diagonal} gives
\[  \left( 2\theta \log \frac{12\theta}{e} \right) n +o(n) \leq \log \frac{\tau((2+\eps)g_n,g_n)}{n^{2g_n}} \leq \left( 2\theta \log \frac{12\theta}{e} + h(\eps) \right) n +o(n), \]
where $h(\eps) \to 0$ as $\eps \to 0$. By combining this with \eqref{eqn_decomposition_taung} and \eqref{eqn_telescopic_product} and finally letting $\eps \to 0$, we get the result.
\end{proof}

We finally prove Proposition \ref{prop_near_diagonal}. The lower bound can be deduced easily from the Goulden--Jackson formula. For the upper bound, we will bound crudely the number of triangulations by the number of tree-rooted triangulations, i.e. triangulations decorated with a spanning tree. These can be counted by adapting classical operations going back to \cite{Mullin67} in the planar case.

\begin{proof}[Proof of Proposition \ref{prop_near_diagonal}]
We start with the lower bound. The Goulden--Jackson formula \eqref{eqn_Goulden_Jackson} implies the following (crude) inequality:
\[\tau(n,g)\geq (36+o(1))n^2 \tau(n-2,g-1),\]
where the $o$ is uniform in $g$, when $n \to +\infty$. By an easy induction, we have
\[\tau((2+\eps)g,g)\geq (36+o(1))^{g}\frac{((2+\eps)g)!}{(\eps g)!}\tau(\eps g,0) \geq (36+o(1))^{g}\frac{((2+\eps)g)!}{(\eps g)!},\]
and the lower bound follows from the Stirling formula.

For the upper bound, we adapt a classical argument about tree-rooted maps. We denote by $\widetilde{\T}((2+\eps)g,g)$ the set of triangulations of genus $g$ with $2(2+\eps)g$ faces and a distinguished spanning tree, and by $\widetilde{\tau}((2+\eps)g,g)$ its cardinal. We recall that such triangulations have $\eps g+2$ vertices, so the spanning tree has $\eps g+1$ edges.

If $t \in \widetilde{\T}((2+\eps)g,g)$, we consider the dual cubic map of $t$ and "cut in two" the edges that are crossed by the spanning tree, as on Figure \ref{fig_mating_tree}. The map that we obtain is unicellular (i.e. it has one face), and we denote it by $U(t)$. Moreover, it is \emph{precubic} (i.e. its vertices have only degree $1$ or $3$), has genus $g$ and $(6+4\eps)g+1$ edges (the number of edges of the original triangulation, plus one for each edge of the spanning tree).
\begin{figure}
\center
\includegraphics[scale=0.8]{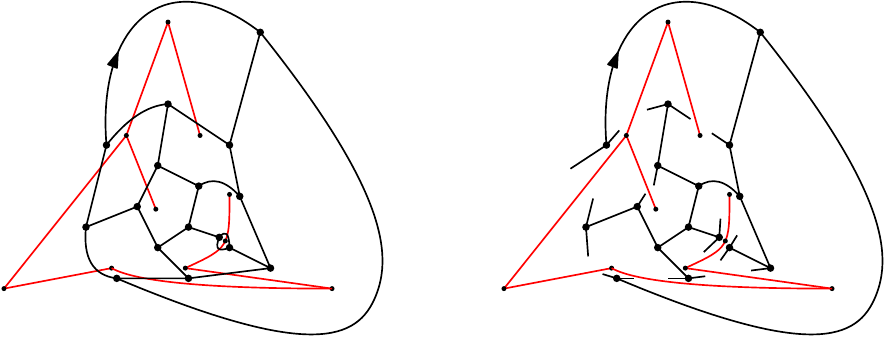}
\caption{An example of the cutting operation on a cubic map. The distinguished spanning tree is in red. We picked a planar example but in the general case the "opened" map on the right is unicellular.}
\label{fig_mating_tree}
\end{figure}
The number of precubic unicellular maps with fixed genus and number of edges is computed exactly in \cite[Corollary 7]{Ch11} and is equal to
\[ \frac{2 ((6+4\eps)g+1)!}{12^g g! (2\eps g+2)! ((3+2\eps)g)!} = e^{h_U(\eps)g+o(g)} \left( \frac{6}{e} \right)^{2g} (2g)^{2g},\]
where
\[h_U(\eps)=2\eps \log \frac{6}{\eps} + (3+2\eps) \log \left( 1+\frac{2 \eps}{3}\right) \xrightarrow[\eps \to 0]{}0 \]
by the Stirling formula. Finally, to go back from $U(t)$ to $t$, we also need to remember how to match the leaves two by two in the face of $U(t)$ without any crossing. The number of ways to do so is $\mathrm{Catalan}(\eps g+1)\leq 4^{\eps g}$, so
\[ \tau((2+\eps)g,g) \leq \widetilde{\tau}((2+\eps)g,g) \leq 4^{\eps g} e^{h_U(\eps)g+o(g)} \left( \frac{6}{e} \right)^{2g} (2g)^{2g},\]
which is enough to conclude. The proof for the second part of the proposition is exactly the same, but where we replace $\eps$ by $\eps_g \to 0$.
\end{proof}

\begin{rem}
Bounding the number of triangulations by the number of tree-rooted triangulations may seem very crude. The reason why this is sufficient is that the spanning trees have only $\eps g+1$ edges, so the number of spanning trees of a triangulation can be bounded by $\binom{3(2+\eps)g}{\eps g+1}$, which is of the form $e^{h(\eps) g+o(g)}$ with $h(\eps) \to 0$ as $\eps \to 0$.
\end{rem}

\bibliographystyle{abbrv}
\bibliography{bibli}

\end{document}